\documentclass[11pt]{article}
\usepackage{amsmath, amsfonts, amsthm,amssymb, color, hyperref, enumerate, extarrows, comment, enumitem}

\newtheorem{theorem}{Theorem}
\newtheorem*{theorem*}{Theorem}

\newtheorem{lem}{Lemma}[section]
\newtheorem{cor}[theorem]{Corollary}
\newtheorem{pro}[lem]{Proposition}
\newtheorem{definition}[lem]{Definition}
\newtheorem{question}[lem]{Question}
\newtheorem*{notation*}{Notation}
{\theoremstyle{definition}
\newtheorem{example}[lem]{Example}}

\newtheorem*{pro*}{Proposition}

\numberwithin{equation}{section}
 {\theoremstyle{definition}
 \newtheorem{remark}[lem]{Remark}}
 \newtheorem*{remark*}{Remark}
\newtheorem{claim}{Claim}[lem]

\newtheorem*{claim*}{Claim}

\newcounter{nmdthmcnt}

\usepackage[T1]{fontenc}

\hypersetup{
	colorlinks   = true,
	citecolor    = magenta}

\newcommand{\bC}{\mathbb{C}}
\newcommand{\bB}{\mathbb{B}}

\begin{document}

\title{\vspace{-1.2cm} \bf  A new strong rigidity phenomenon for the Bergman metric
  \rm}

\author{Peter Ebenfelt, John N. Treuer, Ming Xiao}
\date{}

\maketitle

\begin{abstract}

We establish a new local-to-global rigidity phenomenon for the Bergman
metric. Namely, under natural geometric hypotheses, a local conformal
identification of Bergman metrics determines the underlying complex
manifold globally, up to the unavoidable ambiguity of removing
Bergman-negligible subsets.
More precisely, let $\Omega\subseteq\mathbb{C}^n$ be a bounded domain with a complete Bergman metric, and suppose that the Bergman metric of a complex manifold $M$ is locally conformal, via a holomorphic map $f$, to that of $\Omega$, with a constant conformal factor when $n=1$.
We prove that the given local map $f$ extends to a biholomorphism
$F\colon M\to D$ onto a subdomain $D\subseteq\Omega$ in two complementary
settings. If $M$ is Stein, then $\Omega\setminus D$ is a closed pluripolar
set. If $M$ is a bounded domain and $\Omega$ satisfies a natural symmetry
condition expressed in terms of its automorphism orbits, then
$\Omega\setminus D$ is Bergman-negligible. In particular, this applies when $\Omega$ is a bounded homogeneous domain and yields a characterization, up to Bergman-negligible sets, of a class of bounded domains with locally symmetric Bergman metrics, addressing a question raised by Loi--Palmieri and Zimmer.
A key ingredient in the proofs is a new Calabi-type extension theorem
tailored to Bergman metrics.

\end{abstract}

\renewcommand{\thefootnote}{\fnsymbol{footnote}}
\footnotetext{\hspace*{-7mm}
\begin{tabular}{@{}r@{}p{16.5cm}@{}}
& Keywords.  Bergman metric, Holomorphic Isometry, Stein manifold, biholomorphism\\
& Mathematics Subject Classification. Primary 32H02; Secondary 32A36, 32Q28, 32D15, 32F45
%
\end{tabular}

\noindent\thanks{The first author is supported in part by the NSF grants DMS-2154368 and DMS-2453134. The second author is supported in part by the NSF grant DMS-2247175. The third author is
supported in part by the NSF grants DMS-2045104 and DMS-2554635.}

\noindent \date
}

\section{Introduction}\label{sec:introduction}

Among the fundamental metrics in complex analysis and complex geometry is the Bergman metric. It is biholomorphically invariant and is determined by the $L^2$-holomorphic structure of the underlying manifold $M$. Remarkably, the converse phenomenon also holds to a significant extent:~even the local geometry of the Bergman metric contains a wealth of information about the global holomorphic geometry of $M$. There is, however, an unavoidable limitation. By its very definition, the Bergman metric is insensitive to the removal of certain sufficiently thin subsets. Following the terminology introduced by the present authors in \cite{ETX25}, we call such sets \emph{Bergman-negligible}:

\begin{definition}
    A subset $E$ of a domain $\Omega$ is called {\em Bergman-negligible} if $E$ is closed, measure zero and all functions in the Bergman space of $\Omega \setminus E$ (the Hilbert space of all $L^2$-holomorphic functions on $\Omega \setminus E$, endowed with the $L^2$-inner product) extend holomorphically to $\Omega$.
\end{definition}

Closed pluripolar subsets, such as complex subvarieties, are Bergman-negligible. There are also substantially larger examples; for instance, a sufficiently small closed patch of a real hypersurface may also be Bergman-negligible. We refer to \cite{ETX25} for such an example. By definition, if $\Omega_1\subseteq\Omega_2$ are domains and $\Omega_2\setminus\Omega_1$ is Bergman-negligible in $\Omega_2$, then the Bergman spaces of $\Omega_1$ and $\Omega_2$ are naturally identified, and hence their Bergman metrics agree wherever they are well defined. Thus, one cannot expect the local geometry of the Bergman metric to distinguish such domains.  

The main results of this paper show that, under additional geometric hypotheses, this is essentially the only ambiguity. More precisely, we establish a new rigidity phenomenon: \emph{the conformal class of the Bergman metric of a bounded domain $M$, even when known only on an open subset, determines the global holomorphic geometry of $M$ up to Bergman-negligible sets.}

We next describe the proposed rigidity principle more precisely. Let $M$ be a bounded domain in $\bC^n$ or, more generally, a  complex manifold of dimension $n$ with a well-defined Bergman metric. 
Let $\Omega\subseteq\bC^n$ be a model domain. Here by a \emph{model domain}, we mean a bounded domain whose Bergman metric is complete. This will be a standing assumption on $\Omega$ throughout the paper. Let $\omega_M$ and $\omega_\Omega$ denote the K\"ahler forms of the Bergman metrics of $M$ and $\Omega$, respectively.
Recall that the biholomorphic invariance of the Bergman metric means that if $F\colon M\to\Omega$ is a biholomorphism, then
$
\omega_M=F^*(\omega_\Omega).
$
Motivated by this, we say that the Bergman metrics of $M$ and $\Omega$ are \emph{locally identical} if there exist open sets $U\subseteq M$ and $V\subseteq\Omega$ and a biholomorphism $f\colon U\to V$ such that
$
\omega_M=f^*(\omega_\Omega)
$
on $U$. More generally, we say that their conformal classes are locally identical if there exist such $U$, $V$, and $f$, together with a positive function $\Lambda$ on $U$, such that
$
\omega_M=\Lambda f^*(\omega_\Omega).
$

The main question considered in this paper is whether the global holomorphic geometry of $M$ can be recovered from the germ of the conformal class of its Bergman metric. More precisely, we ask:

\begin{question}\label{MainQuestion}
Assume $n\geq 2$, and let $M$ and $\Omega$ be as above. Suppose that there
exist open connected sets $U\subseteq M$ and $V\subseteq\Omega$, a
biholomorphism $f\colon U\to V$, and a positive function $\Lambda$ on $U$
such that
\begin{equation}\label{Normalizing condition}
\omega_M=\Lambda f^*(\omega_\Omega).
\end{equation}
Does $f$ extend to a biholomorphism
$
F\colon M\longrightarrow D
$
onto a subdomain $D\subseteq\Omega$ such that $\Omega\setminus D$ is
Bergman-negligible in $\Omega$? The question remains valid for $n=1$ if one replaces
the positive function $\Lambda$ by a positive constant $\lambda$.
\end{question}

Note that once one has established an affirmative answer to
Question~\ref{MainQuestion}, it follows that necessarily $\Lambda\equiv 1$,
by the biholomorphic invariance of the Bergman metric and the definition of
Bergman-negligible subsets. Furthermore, when $n\geq 2$, one can show
\emph{a priori} that the conformal factor $\Lambda$ in
\eqref{Normalizing condition} must be a constant $\lambda>0$; see
Proposition~\ref{Conformal implies constant proposition in dimension greater than one}
below. This is no longer true when $n=1$, since in that case all Hermitian
metrics are conformal, which explains the constant-factor formulation in
Question~\ref{MainQuestion} for complex dimension one. Thus, in what
follows, we shall mainly consider \eqref{Normalizing condition} with the
understanding that $\Lambda\equiv\lambda>0$.

\begin{remark}
The completeness assumption on the Bergman metric of $\Omega$ is essential
in Question~\ref{MainQuestion}. For example, let
$M=\mathbb B^2\setminus\{0\}$ and
$\Omega=\mathbb B^2\setminus\{z_1=0\}$. Then the Bergman metrics of $M$
and $\Omega$ agree locally under the identity map; in particular, on any
open connected set $U\subseteq\Omega \subseteq M$:
$
\omega_M=\operatorname{id}^*(\omega_\Omega)
$
on $U$. However, the identity map on $U$ does not extend to a holomorphic
map from $M$ into $\Omega$. Moreover, $M$ is not biholomorphic to any
subdomain $D\subseteq\Omega$ such that $\Omega\setminus D$ is
Bergman-negligible in $\Omega$. Thus the conclusion of
Question~\ref{MainQuestion} can fail when the Bergman metric of $\Omega$
is not complete.
\end{remark}

To the best of our knowledge, Question~\ref{MainQuestion} has not previously
been formulated in this generality. Nevertheless, several important results
in the literature may be viewed, from the perspective of the present paper,
as rigidity phenomena that motivate this question. In his seminal work, Mok \cite{Mo12} established a number of general
rigidity and extension theorems for local holomorphic isometric maps with
respect to Bergman metrics, stimulating many subsequent studies of
holomorphic isometric maps, particularly between bounded symmetric domains.
In particular, it follows from Mok \cite[Theorem~2.1.1]{Mo12} that, in the
setting of Question~\ref{MainQuestion}, if both $M$ and $\Omega$ are bounded
domains, then $f$ extends to a holomorphic isometric immersion of
$(M,\omega_M)$ into $(\Omega,\lambda\omega_\Omega)$. (If, in addition, the
Bergman metric of $M$ is also complete, then $f$ extends to a
biholomorphism from $M$ onto $\Omega$. This shows that the
model domain $\Omega$ is unique up to biholomorphism.)
In light of Mok's work, a principal remaining issue in
Question~\ref{MainQuestion} is therefore to determine whether the complement
$\Omega\setminus D$ of the image of the extended map is
Bergman-negligible.

Another important line of research that motivates our work concerns complex manifolds
whose Bergman metrics have constant holomorphic sectional curvature. In the framework of
Question~\ref{MainQuestion}, this corresponds to the special case in which
the model domain $\Omega$ is the complex unit ball. Indeed, the Bergman
metric of $M$ is locally identified, via a holomorphic map, with
$\lambda$ times the Bergman metric of $\bB^n$ for some constant
$\lambda>0$ if and only if $\omega_M$ has negative constant holomorphic
sectional curvature. The study of this rigidity phenomenon goes back to
the classical theorem of Qi-Keng Lu \cite{L65}, which states that a bounded
domain $M\subseteq\bC^n$ whose Bergman metric is complete and has constant
holomorphic sectional curvature is biholomorphic to the unit ball
$\bB^n\subseteq\bC^n$.  In the setting of incomplete Bergman metrics with constant holomorphic sectional curvature, a fundamental breakthrough was achieved by Huang and Li, who settled in \cite{HL24, HuLi26} a long-standing folklore conjecture by providing a decisive characterization of Stein manifolds whose Bergman metrics have constant holomorphic sectional curvature. In the antecedents by Dong–Wong \cite{DW22a, DW22b}, several extensions of Lu’s theorem to the setting of incomplete Bergman metrics were established using a different method. (See also the work \cite{HX20} of Huang and the third author on Bergman metrics of complex singular spaces.)
Inspired by Huang and Li's work \cite{HL24}, in \cite{ETX25},
the present authors established a rigidity result for the unit-ball model which, in the terminology introduced here, may be viewed as an affirmative answer to the special case $\Omega=\bB^n$ of
Question~\ref{MainQuestion}. See also
\cite{EbTrXi25, BhGaNaXi26,  Pa25, LoPa26} for many other investigations
of Bergman metrics with constant holomorphic sectional curvature and related rigidity questions which have been motivated by Huang--Li's work \cite{HL24}.

The main results of this paper give affirmative answers to
Question~\ref{MainQuestion} in two important independent settings:~one in which $M$ is
assumed to be Stein, and the other in which $\Omega$ possesses a certain
degree of symmetry. These results are established in
Theorem~\ref{The main Stein result} and
Theorem~\ref{Main result - bounded homogoeneous case}, respectively. 

In particular, Theorem~\ref{The main Stein result} is again inspired
by the work of Huang and Li \cite{HL24} and may be viewed as a
natural continuation of their work; their results in
\cite{HuLi26} on the point separation property of Bergman spaces
play an important role in its proof as well. Throughout this paper, all open connected sets are understood to be nonempty.

\begin{theorem}\label{The main Stein result}
Let $M$ be an $n$-dimensional Stein manifold, $n\geq 1$, with a well-defined
Bergman metric $\omega_M$. Let $\Omega\subseteq\mathbb{C}^n$ be a bounded
domain with a complete Bergman metric $\omega_\Omega$. Assume that there
exist an open connected set $U\subseteq M$ and a holomorphic map
$f\colon U\to\Omega$ such that
\begin{equation}\label{Normalizing condition-1}
\omega_M=\lambda f^*(\omega_\Omega)
\end{equation}
for some constant $\lambda>0$. Then $f$ extends to a biholomorphism $F\colon M\to D$ onto a subdomain
$D\subseteq\Omega$ such that $\Omega\setminus D$ is a closed pluripolar
subset of $\Omega$. Moreover, when $n\geq 2$, the constant $\lambda$ in
\eqref{Normalizing condition-1} may be replaced by an everywhere positive
function $\Lambda$ on $U$.
\end{theorem}

\begin{remark}\label{rmk:question21dimension1}
Note that every planar domain, i.e.~every domain in $\mathbb{C}$, is Stein. Theorem~\ref{The main Stein result},
together with its proof, yields an affirmative answer to
Question~\ref{MainQuestion} in complex dimension one:
{\it Let $M$ be a planar domain, or more generally a Stein Riemann surface, or a
Bergman-nondegenerate Riemann surface, with a well-defined Bergman metric
$\omega_M$. Let $\Omega$ be a bounded planar domain with a complete
Bergman metric $\omega_\Omega$. Assume that there exist an open connected
set $U\subseteq M$ and a holomorphic map $f\colon U\to\Omega$ such that
\eqref{Normalizing condition-1} holds. Then $f$ extends to a biholomorphism $F\colon M\to D$ onto a subdomain
$D\subseteq\Omega$ such that $\Omega\setminus D$ is a closed polar subset
of $\Omega$.} We will provide a justification for this statement in \S\ref{Section Theorem 1 proved}.
\end{remark}

Having resolved Question~\ref{MainQuestion} in complex dimension one, we
focus in what follows on the case $n\geq 2$.
In our second main result, we impose no Stein assumption on $M$. Instead, we impose a symmetry condition on the model domain
$\Omega$. Throughout the paper, for a domain $\Omega$, let
$\operatorname{Aut}(\Omega)$ denote its group of holomorphic automorphisms.
The $\operatorname{Aut}(\Omega)$-orbit through a point $p\in\Omega$ is the
set of images of $p$ under all automorphisms of $\Omega$. We say that the $\operatorname{Aut}(\Omega)$-orbit through $p\in\Omega$
contains a nonconstant holomorphic curve if it contains the image of a
nonconstant holomorphic map $\gamma\colon\mathbb D\to\Omega$. Here $\mathbb{D}$ is the unit disk in $\mathbb{C}$.

We also recall that a subset $X\subseteq\Omega$ is said to be \emph{analytically Zariski dense} in $\Omega$ if it is not contained in any proper complex-analytic subset of $\Omega$, or equivalently, if its analytic Zariski closure is all of $\Omega$. When $\Omega$ is Stein, Cartan's theorem~A implies that this is equivalent to requiring that every holomorphic function on $\Omega$ that vanishes on $X$ vanish identically on $\Omega$.
Typical examples include
any nonempty open subset of $\Omega$ and any nonempty open piece
of a smooth real hypersurface contained in $\Omega$.
We are now ready to state our second main result.

\begin{theorem}\label{Main result - bounded homogoeneous case}
Let $M\subseteq\mathbb{C}^n$, $n\geq 2$, be a bounded domain with a Bergman
metric $\omega_M$, and let $\Omega\subseteq\mathbb{C}^n$ be a bounded domain
with a complete Bergman metric $\omega_{\Omega}$. Suppose there exists an
analytically Zariski dense subset $X\subseteq\Omega$ such that, for every
$p\in X$, the $\operatorname{Aut}(\Omega)$-orbit through $p$ contains a nonconstant
holomorphic curve. Assume further that there exist an open connected set
$U\subseteq M$ and a holomorphic map $f:U\to\Omega$ such that
\begin{equation}\label{Normalizing condition-2}
\omega_M=\Lambda f^*(\omega_\Omega),
\end{equation}
where $\Lambda$ is a positive function on $U$. Then $f$ extends to a biholomorphism $F\colon M\to D$ onto a subdomain
$D\subseteq\Omega$ such that $\Omega\setminus D$ is a
Bergman-negligible subset of $\Omega$.
\end{theorem}

The hypotheses imposed on the target $\Omega$ are satisfied, in particular,
by bounded homogeneous domains. More generally, they are also satisfied
when $\Omega$ is the product of a bounded homogeneous domain and another
bounded domain with complete Bergman metric. In both cases, one may take $X=\Omega$, since for every $p\in\Omega$, the orbit of $p$ under $\operatorname{Aut}(\Omega)$ contains a nonconstant holomorphic curve. We therefore obtain the following consequence.

\begin{cor}\label{The new main theorem 2s corollary is the bounded homogeneous domain}
Let $M\subseteq\mathbb{C}^n$, $n\geq 2$, be a bounded domain with a Bergman
metric $\omega_M$. Let $\Omega\subseteq\mathbb{C}^n$ be either a bounded
homogeneous domain or the product of a bounded homogeneous domain and
another bounded domain with a complete Bergman metric. Let $\omega_{\Omega}$
denote the Bergman metric of $\Omega$. Assume that there exist an open
connected set $U\subseteq M$ and a holomorphic map $f:U\to\Omega$ such that
\begin{equation}\label{Normalizing condition-222}
\omega_M=\Lambda f^*(\omega_\Omega),
\end{equation}
where $\Lambda$ is a positive function on $U$. Then $f$ extends to a biholomorphism $F\colon M\to D$ onto a subdomain
$D\subseteq\Omega$ such that $\Omega\setminus D$ is a
Bergman-negligible subset of $\Omega$.
\end{cor}

Specializing the above corollary to the case where the model domain
$\Omega$ is a bounded symmetric domain, we answer a question of
Loi--Palmieri \cite{LoPa26} and Zimmer \cite{Zi25} concerning the
characterization, up to Bergman-negligible sets, of bounded domains whose
Bergman metrics are locally symmetric (see \S\ref{Subsection where Theorem 5 is proved} or
\cite{LoPa26} for the definition).

\begin{cor}\label{Corollary 8 - locally symmetric Bergman metrics and the image is a domain minus a Bergman negligible subset}
Let $M\subseteq\mathbb{C}^n$, $n\geq 2$, be a bounded domain whose Bergman
metric is locally irreducible and locally symmetric. Then there exists an irreducible bounded symmetric domain
$\Omega\subseteq\mathbb{C}^n$ such that $M$ is biholomorphic to
$\Omega\setminus E$ for some Bergman-negligible subset $E\subseteq\Omega$.
\end{cor}

For earlier and related work on bounded pseudoconvex domains with locally symmetric Bergman metrics, see Loi--Palmieri \cite{LoPa26}.

The conclusion of Corollary~\ref{Corollary 8 - locally symmetric Bergman metrics and the image is a domain minus a Bergman negligible subset}
does not extend to arbitrary unbounded domains, or to
arbitrary Bergman-nondegenerate complex manifolds. Indeed, Huang--Li
\cite{HL24} constructed the first, and later the authors of \cite{HK26}
and \cite{BhGaNaXi26} constructed more examples of Bergman-nondegenerate unbounded domains whose Bergman metrics are
locally isometric to positive multiples of the Fubini--Study metric.
These metrics are locally irreducible and locally symmetric, but their
positive holomorphic sectional curvature rules out the conclusion of
the corollary.
We also remark that the local irreducibility assumption in
Corollary~\ref{Corollary 8 - locally symmetric Bergman metrics and the image is a domain minus a Bergman negligible subset}
can be removed using ideas developed in this paper, provided that
$\Omega$ is allowed to be reducible. Since the proof is more involved
and goes beyond the scope of this paper, we will present it in a
separate article.

We next briefly discuss some of the difficulties in proving
Theorems~\ref{The main Stein result} and
\ref{Main result - bounded homogoeneous case}, and compare with proofs in earlier
works. As mentioned above, closely related rigidity problems for the
unit-ball model $\Omega=\bB^n$ have been extensively studied in
\cite{L65, HX20, DW22a, DW22b, HL24, ETX25, EbTrXi25}. In almost all of
these works, the explicit formula and special properties of the Bergman
kernel $K_{\bB^n}$ played a fundamental role. Such an explicit formula is
generally unavailable for a model domain $\Omega$. Moreover, as used in \cite{HX20, HL24, EbTrXi25}, another important feature
of the Bergman metric $\omega_{\bB^n}$ is that, for every positive constant
$\lambda$, the K\"ahler manifold $(\bB^n,\lambda\omega_{\bB^n})$ admits a
holomorphic isometric embedding into the infinite-dimensional projective
space $\mathbb{P}^{\infty}$; see Section~\ref{Section Preliminaries} for
the relevant definition. This can be seen directly from the power-series
expansion of $K_{\bB^n}^{\lambda}$, which can be expressed as an infinite
sum of squared moduli of holomorphic functions. The existence of such an embedding makes Calabi's rigidity theorem
\cite{Ca53} for K\"ahler immersions particularly effective in this setting. This embeddability property, however, need not hold
for a general model domain $\Omega$; indeed, it can fail even for bounded
symmetric domains of higher rank. The following example illustrates this
phenomenon.

\begin{example}\label{eg1}
Let $\lambda$ be a positive constant, and
let $\Omega$ be the type I classical domain in $M(2,2,\mathbb C):$ 
$$
\Omega=D^{\mathrm I}_{2,2}
=
\left\{
Z\in M(2,2,\mathbb C): I_2-ZZ^*>0
\right\},
$$
where $M(2,2,\mathbb C)$ denotes the space of $2\times2$ complex matrices
and $I_2$ denotes the $2\times2$ identity matrix. Let $\omega_\Omega$
denote the Bergman metric of $\Omega$. 
By a result of Loi--Zedda \cite[Theorem 2]{LoiZedda11}, $(\Omega,\lambda\omega_\Omega)$ admits a holomorphic isometric immersion
into $\mathbb P^\infty$ if and only if
$\lambda\geq\frac14$. 
Thus, unlike the unit ball, an arbitrary positive multiple of the Bergman
metric of a higher-rank bounded symmetric domain need not admit a
holomorphic isometric immersion into $\mathbb P^\infty$. More general
examples of this phenomenon, for irreducible bounded symmetric domains,
can be found in Loi--Zedda \cite{LoiZedda11}.
\end{example}

To overcome these difficulties, we develop a new Calabi-type extension theorem (see Theorem~\ref{Theorem 2 and 2+ in Ming's Notes}) that is well adapted to the geometry of the Bergman metric. This approach is inspired by Mok's work \cite{Mo12} on the extension and rigidity of local holomorphic isometries. We believe that this Calabi-type extension theorem, together with the accompanying continuation method, may be useful more broadly in the study of Bergman isometric mapping problems. In particular, we will establish the following fact.

\begin{remark}
\label{rmk:non-projectively-induced-target}
Let $\Omega \subseteq \mathbb{C}^n$ be a bounded domain with Bergman metric $\omega_\Omega$. Let $\lambda$ be a positive constant such that $(\Omega,\lambda\omega_\Omega)$ does not admit a holomorphic isometric embedding into $(\mathbb{P}^{\infty}, \omega_{FS})$.
Let $M$ be an $n$-dimensional complex manifold with a well-defined Bergman metric $\omega_M$. As a byproduct of the proofs and methods developed in this paper, for any nonempty open set $U\subseteq M$, there cannot exist a holomorphic map $f:U\to\Omega$ such that $\omega_M=\lambda f^*(\omega_\Omega)$ on $U$.
This is proved in Proposition~\ref{Lemma prior to the main Stein theorem}.
\end{remark}

Beyond the difficulties discussed above, there is a main difficulty specific to
Theorem~\ref{Main result - bounded homogoeneous case} due to the presence
of the conformal factor $\Lambda$. As noted above, when $n\geq 2$, the
K\"ahler condition implies that $\Lambda$ must be a positive constant
$\lambda$. A major step in the proof of
Theorem~\ref{Main result - bounded homogoeneous case} is then to show that
$\lambda=1$, thereby reducing the problem to the case of an exact local
isometry. By contrast, in Theorem~\ref{The main Stein result}, the Stein
assumption on $M$ allows us to bypass this  step; there is no
need to first prove that $\lambda=1$. As a result, the proof of
Theorem~\ref{Main result - bounded homogoeneous case} is more involved than
that of Theorem~\ref{The main Stein result}.

In the setting of Theorem~\ref{Main result - bounded homogoeneous case},
once we know that $\lambda=1$, so that $f$ is an exact local isometry, the
rigidity conclusion in fact holds in considerably greater generality, as
shown by the following theorem.


\begin{theorem}\label{Lambda = 1 Theorem}
Let $M\subseteq\mathbb{C}^n$, $n \geq 1,$ be a bounded domain with a Bergman
metric $\omega_M$. Let $\Omega\subseteq\mathbb{C}^n$ be a bounded domain
with a complete Bergman metric $\omega_\Omega$. Assume that there exist an
open connected set $U\subseteq M$ and a holomorphic map
$f\colon U\to\Omega$ such that
\begin{equation}\label{Exact condition}
f^*(\omega_\Omega)=\omega_M.
\end{equation}
Then $f$ extends to a biholomorphism $F\colon M\to D$ onto a subdomain
$D\subseteq\Omega$ such that $\Omega\setminus D$ is a
Bergman-negligible subset of $\Omega$.
\end{theorem}



To provide further evidence for an affirmative answer to Question~\ref{MainQuestion}, we next consider cases not covered by Theorem~\ref{Main result - bounded homogoeneous case}.  Of particular interest is the case in which $\Omega$ is a Hartogs domain over a bounded homogeneous base. For such domains, the analytic subvarieties contained in automorphism orbits may be too small for Theorem~\ref{Main result - bounded homogoeneous case} to apply.  Explicit examples are provided by the nonhomogeneous egg domains
\begin{equation}\label{Egg domain equation}
\{(z, \zeta) \in \mathbb{B}^n \times \mathbb{C}^m: \|z\|^2 + \|\zeta\|^{2p} < 1\}, \quad p > 0, \quad p \neq 1,
\end{equation}
for which $\mathbb{B}^n \times \{0\}$ is the only automorphism orbit containing a nontrivial holomorphic curve (see Remark~\ref{Final remark saying the egg domain corollary only follows from the last theorem}).  In our final theorem, we use the machinery developed in this paper to resolve Question~\ref{MainQuestion} for a natural class of Hartogs domains over bounded homogeneous bases.

\begin{theorem}\label{Thm:Main Hartogs Theorem}
Let $M\subseteq\mathbb{C}^{n + m}$, $n + m\geq 2$, be a bounded domain with a Bergman
metric $\omega_M$. 
Let $B\subseteq\mathbb C^n$ be a bounded homogeneous domain, and let $\psi:B\to\mathbb R$ be a smooth strictly plurisubharmonic function such that the associated K\"ahler metric $\omega_{\psi}=i\partial\overline\partial\psi$ is homogeneous. Define the Hartogs domain
\begin{equation}\label{Eqn:Description of the Hartogs domain in Theorem 1.1}
\Omega=\left\{(z,\zeta)\in B\times\mathbb C^m:
t(z,\zeta):=e^{\psi(z)}\|\zeta\|^2<1\right\},
\end{equation}
and assume that $\Omega$ is bounded. Write $\omega_\Omega$ for its Bergman metric.
Suppose that there exist an open connected set
$U\subseteq M$ and a holomorphic map $f:U\to\Omega$ such that
$\omega_M=\Lambda f^*(\omega_\Omega),$
where $\Lambda$ is a positive function on $U$. Then $f$ extends to a biholomorphism $F\colon M\to D$ onto a subdomain
$D\subseteq\Omega$ such that $\Omega\setminus D$ is a
Bergman-negligible subset of $\Omega$.
\end{theorem}
In applications, the potential function $\psi$ in Theorem \ref{Thm:Main Hartogs Theorem} is often chosen to be
\begin{equation}\label{natural choice for psi is log Bergman kernel}
\psi=s_1\log (s_2K_B),
\qquad \hbox{for some  constants }s_1, s_2>0,
\end{equation}
where $K_B$ is the Bergman kernel of $B$. 
With this choice of $\psi,$ since the Bergman metric is invariant under biholomorphisms, $i\partial\overline\partial\psi$ is homogeneous.  Theorem \ref{Thm:Main Hartogs Theorem} therefore applies whenever the resulting total Hartogs domain is bounded. The egg domains in \eqref{Egg domain equation} can be rewritten in the format of \eqref{Eqn:Description of the Hartogs domain in Theorem 1.1} using such a $\psi$ as in \eqref{natural choice for psi is log Bergman kernel}.  Therefore, Theorem \ref{Thm:Main Hartogs Theorem} resolves Question \ref{MainQuestion} when the model is such an egg domain.

\begin{cor}\label{Egg domain corollary}
Let $M \subseteq \mathbb{C}^{n + m}$, $n  + m\geq 2$, be a bounded domain with a Bergman metric $\omega_M$. For some $p > 0$ with $p \neq 1$, let $\Omega$
be an egg domain as defined in \eqref{Egg domain equation}.  Assume that there exist an open
connected set $U\subseteq M$ and a holomorphic map $f:U\to\Omega$ such that
$\omega_M=\Lambda f^*(\omega_\Omega),$
where $\Lambda$ is a positive function on $U$. Then $f$ extends to a biholomorphism $F\colon M\to D$ onto a subdomain $D\subseteq\Omega$, with $\Omega\setminus D$ Bergman-negligible.
\end{cor}

We conclude this section with the following remark.

\begin{remark}\label{Remark where we say that everything holds if it is also a Bergman nondegenerate complex manifold}
Bounded domains in $\mathbb{C}^n$ are examples of Bergman-nondegenerate
complex manifolds, that is, complex manifolds whose Bergman spaces separate
points and induce nondegenerate Bergman metrics; see Section~\ref{Section Preliminaries} for
the precise definition. Theorems~\ref{Main result - bounded homogoeneous case},
\ref{Lambda = 1 Theorem} and \ref{Thm:Main Hartogs Theorem} as well as
Corollaries~\ref{The new main theorem 2s corollary is the bounded homogeneous domain} and \ref{Egg domain corollary}
remain valid under the weaker assumption that $M$ is a
Bergman-nondegenerate complex manifold; see their proofs in \S\ref{Subsection where Theorem 5 is proved} and \S\ref{Sec: The last section for the hartogs domains over bounded homogeneous bases}. In particular,
Theorem~\ref{Main theorem bounded homogeneous case - but with M a complex manifold that separates points} in \S\ref{Subsection where Theorem 5 is proved}
 gives the version of
Theorem~\ref{Main result - bounded homogoeneous case} under this weaker
hypothesis. 

\end{remark}

The paper is organized as follows. In \S\ref{Section Preliminaries}, we
present the background needed for the proofs. In
\S\ref{sec: new calabi thm}, we develop a new Calabi-type extension
theorem. In \S\ref{Section Theorem 1 proved}, we prove
Theorem~\ref{The main Stein result}. As preparation for the proof of
Theorem~\ref{Main result - bounded homogoeneous case}, in
\S\ref{Section that proves the bounded homogeneous case} we establish several preliminary
results. In
\S\ref{Subsection where Theorem 5 is proved}, we prove
Theorem~\ref{Main result - bounded homogoeneous case} and its corollaries,
as well as Theorem~\ref{Lambda = 1 Theorem}.  Finally, in \S\ref{Sec: The last section for the hartogs domains over bounded homogeneous bases}, we prove Theorem \ref{Thm:Main Hartogs Theorem} and its corollary.
\medskip

\textbf{Acknowledgements: }Before posting this paper on arXiv, we learned that Yuan had independently
obtained some similar results in his preprint \cite{Yu26}, especially in
connection with Theorem~\ref{The main Stein result} (see also Lemma~\ref{Proposition that shows that Omega minus D is a closed pluripolar subset}) and
Theorem~\ref{Lambda = 1 Theorem} (see also Proposition~\ref{Theorem I'}) of the present paper. Our approaches differ substantially. Yuan's proof uses Irgens' $L^2$-envelope of holomorphy, whereas our proof relies on a new Calabi-type extension theorem that we develop in this paper.


\section{Preliminaries}\label{Section Preliminaries}
\subsection{The Bergman space, kernel and metric}
Let $M$ denote an $n$-dimensional complex manifold. Throughout the paper,
all complex manifolds are assumed to be connected. The Bergman space of
$M$, denoted by $A^2_{(n,0)}(M)$, is the Hilbert space of holomorphic
$(n,0)$-forms with finite $L^2$-norm, with inner product
$$
\langle \Phi,\Psi\rangle
=
\frac{i^{n^2}}{2^n}\int_M\Phi\wedge\overline{\Psi}.
$$
Since $A^2_{(n,0)}(M)$ is separable, it admits an orthonormal basis
$\{\Phi_k\}_{k=1}^N$, where
$N=\dim A^2_{(n,0)}(M)$ may be finite or countably infinite. The Bergman kernel
of $M$ is defined by
\begin{equation}\label{Orthonormal series representation}
K_M
=
\sum_{k=1}^N \Phi_k\wedge\overline{\Phi_k}.
\end{equation}
This definition is independent of the choice of orthonormal basis.
In local coordinates $z=(z_1,\ldots,z_n)$ and
$w=(w_1,\ldots,w_n)$, the Bergman kernel can be written as
$$
K_M(z,w)
=
K_M^*(z,w)\,
dz_1\wedge\cdots\wedge dz_n
\wedge
d\overline{w}_1\wedge\cdots\wedge d\overline{w}_n.
$$
In our arguments on general complex manifolds, whenever the representing
function of the Bergman kernel is used, we work in a fixed coordinate
chart and, by an abuse of notation, write $K_M$ for $K_M^*$. Henceforth,
we will not use the notation $K_M^*$.
Under this convention, $K_M(z,w)$ is holomorphic in $z$ and
anti-holomorphic in $w$, and \eqref{Orthonormal series representation}
gives $K_M(z,z)\geq0$. We will frequently restrict the Bergman kernel to
the diagonal and write $K_M(z):=K_M(z,z)$, or suppress the arguments
entirely when there is no ambiguity. In particular, $\log K_M$ will
always mean $\log K_M(z,z)$.

For a domain $\Omega\subseteq\mathbb C^n$, let $A^2(\Omega)$ denote the
Hilbert space of square-integrable holomorphic functions on $\Omega$.
We will use the Euclidean coordinates to naturally identify $A^2(\Omega)$ with
$A^2_{(n,0)}(\Omega)$, and identify the Bergman
kernel with its representing function. We will use the reproducing
property
$$
\phi(z)
=
\int_\Omega \phi(w)K_\Omega(z,w)\,dV(w),
\qquad \phi\in A^2(\Omega).
$$
Here and throughout the paper, $dV$ denotes the Euclidean volume element.
We will also use the transformation law for the Bergman kernel. Namely, if
$F\colon\Omega\to\Omega'$ is a biholomorphism, then
$$
K_\Omega(z,w)
=
\det JF(z)\,
\overline{\det JF(w)}\,
K_{\Omega'}(F(z),F(w)),
$$
where $JF$ denotes the complex Jacobian matrix of $F$.

For a complex manifold $M$, its Bergman space is said to be base-point free if $K_M(z, z) > 0$ for every point $z \in M$.  Its Bergman space is said to separate holomorphic directions if for every $z \in M$ and every nonzero $(1, 0)$ vector $Z$, there is a $g \in A^2_{(n, 0)}(M)$ such that its representative $g^*$ in local coordinates vanishes at $z$, but $(Zg^*)|_{z} \neq 0$.  Any complex manifold $M$ has a Bergman space that satisfies these two properties if and only if it admits a well-defined (nondegenerate) K\"ahler metric called the Bergman metric, \cite{K59}.  Locally, the K\"ahler $(1, 1)$-form $\omega_M$ for the Bergman metric is defined by
$$
\omega_M = i\partial\overline{\partial}\log K_M,
$$
where once again $K_M$ has been identified with its local representative.  The Bergman metric is a biholomorphically invariant metric, meaning that if $f:M \to N$ is a biholomorphism between complex manifolds admitting Bergman metrics, then $f^*\omega_N = \omega_M.$  

A Bergman space is said to separate points if for any two distinct points $z_1$ and $z_2$, there is an $(n, 0)$-form $f \in A^2_{(n, 0)}(M)$ such that $f$ vanishes at $z_1$, but not at $z_2$.  A complex manifold $M$ that admits a Bergman metric and whose Bergman space separates points is called Bergman-nondegenerate. Every bounded domain in $\mathbb C^n$ is Bergman-nondegenerate. Such manifolds were considered extensively in \cite{HL24, ETX25, BhGaNaXi26} in the classification of Bergman metrics of constant holomorphic sectional curvature.


\subsection{Some auxiliary classical results}\label{sec: some auxiliary results}
The Bergman kernel on the diagonal can always be locally expressed as a sum
of squared moduli of holomorphic functions. This feature of its definition
naturally leads one to consider  holomorphic isometric immersions of a
complex manifold endowed with its Bergman metric into projective space
equipped with the Fubini--Study metric.  Let
$$
\ell^2=\left\{(z_1,z_2,\ldots):\sum_{j=1}^{\infty}|z_j|^2<\infty\right\}.
$$
We say that a map $F$ from a complex manifold $M$ to $\ell^2$ is
holomorphic if, for every $p\in M$, there exist a neighborhood $W$ of $p$
and a sequence of holomorphic functions $\{f_k\}_{k=1}^{\infty}$ on $W$
such that $\sum_{k=1}^{\infty}|f_k|^2$ converges uniformly on compact
subsets of $W$ and
$$
F=(f_1,f_2,\ldots)\quad\text{on }W.
$$
We further define, for $2\leq N\leq\infty$,
$$
\mathbb P^{N-1}
=
\begin{cases}
(\mathbb C^N\setminus\{0\})/\sim, & N<\infty,\\
(\ell^2\setminus\{0\})/\sim, & N=\infty,
\end{cases}
$$
where ${\bf x}\sim{\bf y}$ if and only if there exists
$\lambda\in\mathbb C^*$ such that ${\bf x}=\lambda{\bf y}$.
When $N=\infty$, we interpret $\mathbb P^{N-1}$ as
$\mathbb P^\infty$ throughout the paper. The Fubini--Study metric on $\mathbb P^{N-1}$, understood formally when
$N=\infty$, is given in homogeneous coordinates $[z_1,z_2,\ldots]$ by
$$
\omega_{FS} = i\partial\overline{\partial}\log(\sum_{j=1}^N |z_j|^2).
$$
A  holomorphic isometric immersion into $\mathbb{P}^{N - 1}$ is defined as follows:

\begin{definition}
    A map $\mathcal{F}$ from a K\"{a}hler manifold $(M, \omega)$ to $(\mathbb{P}^{N-1}, \omega_{FS})$, $2 \leq N \leq \infty$, is called a holomorphic isometric immersion if for all $p \in M$, there exist a neighborhood $W$ of $p$ and a sequence of holomorphic functions $\{f_k\}_{k=1}^{N}$ on $W$ such that
    \begin{enumerate}
        \item The $f_k$'s have no common zeros on $W$ and $\sum_{k=1}^{N} |f_k|^2$ converges uniformly on $W$ when $N=\infty;$
        \item  It holds that
        $$
        i\partial\overline{\partial}\log \sum_{k=1}^{N} |f_k|^2 = \omega, \quad \hbox{ on $W$;}
        $$
        \item On $W$,
        $$
        \mathcal{F} = [F] := \begin{cases}
             [f_1, \ldots, f_N]& N < \infty\\
             [f_1, f_2, \ldots]& N = \infty.
             \end{cases}
             $$

    \end{enumerate}
    The map $F = (f_1, \ldots, f_N): W \to \mathbb{C}^N$ or respectively $F = (f_1, \ldots, f_n, \ldots): W \to \ell^2$ is called a representation of $\mathcal{F}$ on $W$. If, in addition, $\mathcal{F}$ is injective, we call it a holomorphic
isometric embedding.
\end{definition}

The  holomorphic isometric immersion from a complex manifold $M$ with a Bergman metric $\omega_M$ into projective space we will consider most frequently is the Bergman-Bochner map $\beta_M$.  Given an orthonormal basis $\{\Phi_j\}_{j=1}^N$, with local representatives $\{\phi_j\}_{j=1}^N$ in some coordinate chart, the Bergman-Bochner map is defined locally by
$$
\beta_M =
\begin{cases}
    [\phi_1,\ldots, \phi_N]& N < \infty \\
    [\phi_1, \phi_2, \ldots]& N = \infty.
\end{cases}
$$
The Bergman-Bochner map $\beta_M$ is injective if and only if the Bergman space $A^2_{(n, 0)}(M)$ separates points. 
We recall the following rigidity and extension theorem of Calabi \cite{Ca53} for holomorphic isometric immersions into projective space, in the formulation of \cite{HuLi25, HL12}.
\begin{theorem}[Calabi's rigidity theorem, \cite{Ca53, HuLi25, HL12}]\label{thm:Rigidity}
Let $(M,\omega)$ be a K\"ahler manifold, and
$F_1,F_2\colon (M,\omega)\to(\mathbb P^{N-1},\omega_{FS})$,
$2\leq N\leq\infty$, be two holomorphic isometric immersions.
For $j=1,2$, let $H_j$ be the smallest closed projective linear
subspace of $\mathbb P^{N-1}$ containing $F_j(M)$, equipped with the
induced Fubini--Study metric. Then there exists a projective linear
isometric isomorphism $T\colon H_1\to H_2$ such that
$F_2=T\circ F_1$.
\end{theorem}

\begin{theorem}[Calabi's extension theorem, \cite{Ca53, HuLi25, HL12}]\label{Calabi Extension Theorem}
Let $(M,\omega)$ be a K\"ahler manifold with real-analytic metric
$\omega$. Let $U\subseteq M$ be an open connected set containing $p$,
and let
$F\colon (U,\omega)\to(\mathbb P^{N-1},\omega_{FS})$,
$2\leq N\leq\infty$, be a holomorphic isometric immersion. Then, for
any continuous curve $\gamma\colon[0,1]\to M$ with $\gamma(0)=p$, $F$
admits a holomorphic continuation along $\gamma$ as a holomorphic
isometric immersion into $(\mathbb P^{N-1},\omega_{FS})$.
Moreover, if $M$ is simply connected, then $F$ extends to a globally
defined holomorphic isometric immersion
$F\colon(M,\omega)\to(\mathbb P^{N-1},\omega_{FS})$.
\end{theorem}


\medskip

A tool that will be needed in Theorem \ref{Main result - bounded homogoeneous case}, when $M$ is not necessarily Stein, is a uniqueness theorem from the complex moment problem: Let $\{s_{\alpha\beta}\}_{\alpha, \beta \in \mathbb{N}^n}$ be a sequence of numbers; if there exists a nonnegative Borel measure $\mu$ with compact support such that
$$
\int_{\mathbb{C}^n}z^{\alpha}\bar{z}^{\beta} d\mu = s_{\alpha\beta}, \quad \alpha, \beta \in \mathbb{N}^n,
$$
then $\mu$ is unique. For more detailed explanations that expand on these preliminaries, we refer the reader to the survey portion of the article \cite{EbTrXi25}.

\subsection{Reduction of the conformal factor to a normalizing constant}\label{Section Theorem 2 proof}
Finally, in this subsection, we prove that when $n \geq 2$, the conformal factor $\Lambda$ in  Question~\ref{MainQuestion} is necessarily a constant.
This result should be well-known to experts, but we include a proof for the reader's convenience.

\begin{pro}\label{Conformal implies constant proposition in dimension greater than one}
    Let $M$ be a complex manifold of dimension $n$.  Let $\omega_1$, $\omega_2$ be two K\"{a}hler metrics on $M$ that are conformal to each other; that is, there exists an everywhere positive function $\Lambda$ on $M$ such that
    \begin{equation}
        \omega_2 = \Lambda\omega_1.
    \end{equation}  If $n \geq 2$, then $\Lambda \equiv \lambda$ for some constant $\lambda > 0$.
\end{pro}

\begin{proof}
Since every K\"ahler metric is $C^{\infty}$, the hypotheses imply that $\Lambda$ is $C^{\infty}$. Notice that
$$
d\omega_2 = d\Lambda \wedge \omega_1 + \Lambda d\omega_1 \quad \hbox{ on $M$}.
$$
Since the K\"ahler forms $\omega_1$ and $\omega_2$ are $d$-closed,
$d\Lambda \wedge \omega_1 = 0$ on $M$.
By comparing types, $\partial \Lambda \wedge \omega_1 = 0$ and $\overline{\partial}\Lambda \wedge \omega_1 = 0$.  Pick a local coordinate neighborhood and write without loss of generality that $\omega_1 = \sqrt{-1}\sum_{i, j = 1}^ng_{i\bar{j}}dz_i \wedge d\bar{z}_j$.  We will also use the notation $\Lambda_i = {\partial \Lambda \over \partial z_i}$.  Then $\partial \Lambda \wedge \omega_1 = 0$ yields that
$$
\Lambda_i g_{j\bar{k}} = \Lambda_jg_{i\bar{k}} \quad 1 \leq i, j \leq n, \quad 1 \leq k \leq n.
$$
Taking the trace using $g^{\bar{k}j}$ gives
$$
n\Lambda_i = \sum_{j}\Lambda_j\delta_i^j = \Lambda_i.
$$
Since $n \geq 2$, we conclude $\Lambda_i = 0$.  Since $\Lambda$ is real and everywhere positive,  $\Lambda$ is a positive constant in this coordinate neighborhood.  It follows that globally $\Lambda$ is a positive constant $\lambda > 0$.
\end{proof}

\begin{remark}
The proposition fails when $n=1$. Indeed, on a Riemann surface, every Hermitian metric is K\"ahler. Hence, if $\omega$ is
a Hermitian metric and $\Lambda$ is any positive nonconstant smooth
function, then $\Lambda\omega$ is again a K\"ahler metric. 
\end{remark}

\section{A new Calabi-type extension theorem}\label{sec: new calabi thm}

As mentioned in $\S$\ref{sec:introduction}, to overcome various obstacles
arising in the study of Question~\ref{MainQuestion} in its general setting, inspired by Mok \cite{Mo12},  we develop a new Calabi-type
extension theorem. 
This theorem is well adapted to the geometry of the
Bergman metric even when there is no a priori holomorphic isometric immersion
of $(\Omega,\lambda\omega_\Omega)$ into $\mathbb P^\infty$.
To state our Calabi-type extension theorem, we first introduce the
following definition. Recall that holomorphic maps into $\ell^2$ were defined in
$\S$\ref{sec: some auxiliary results}.

\begin{definition}
Let $U$ be an open connected subset of a complex manifold $M$, and let
$H\colon U\to\mathcal T$ be a holomorphic map, where
$\mathcal T=\mathbb C^N$, $N<\infty$, or $\mathcal T=\ell^2$. Write
$$
H=
\begin{cases}
(h_1,\ldots,h_N)\colon U\to\mathbb C^N, & \mathcal T=\mathbb C^N,\\
(h_1,h_2,\ldots)\colon U\to\ell^2, & \mathcal T=\ell^2.
\end{cases}
$$

\begin{enumerate}[label=(\arabic*)]
\item We say that $H$ admits a (possibly multi-valued) holomorphic extension to
$M$ if, for every continuous curve $\gamma\colon[0,1]\to M$ with
$\gamma(0)\in U$ and every $t\in[0,1]$, there exists a neighborhood
$U_t$ of $\gamma(t)$ such that every component $h_j$ admits an analytic
continuation along $\gamma|_{[0,t]}$ to a holomorphic function $(h_j)_t$
on $U_t$. In the case $\mathcal T=\ell^2$, we further require that
$\sum_{j=1}^{\infty}|(h_j)_t|^2$ converges locally uniformly on $U_t$.

\item We say that $H$ admits a semi-single-valued holomorphic extension to
$M$ if it admits a (possibly multi-valued) holomorphic extension to $M$ and,
for any two continuous curves $\gamma_1,\gamma_2$ in $M$ connecting a
point $p\in U$ to a point $q\in M$, the corresponding analytic
continuations $H_1$ and $H_2$ satisfy
$$
H_1=e^{i\theta}H_2
$$
in a neighborhood of $q$, for some $\theta\in\mathbb R$.

\item We further say that $H$ admits a single-valued holomorphic extension to $M$
if the analytic continuation of $H$ to every point $q\in M$ is independent
of the choice of curve along which it is continued.
\end{enumerate}
\end{definition}

Our Calabi-type extension theorem can now be stated as follows. As mentioned earlier, throughout the  paper, we interpret
$\mathbb P^{N-1}$ as $\mathbb P^\infty$ when $N=\infty$.

\begin{theorem}\label{Theorem 2 and 2+ in Ming's Notes}
Let $(M,\omega)$ be a K\"ahler manifold such that there exists an
everywhere positive real-analytic function $\rho$ on $M$ satisfying
$\omega=i\partial\overline{\partial}\log\rho$. Let $U$ be an open
connected subset of $M$, let $1 \leq N\leq\infty$, and let
$\{h_j\}_{j=1}^N$ be holomorphic functions on $U$ such that
$$
\sum_{j=1}^N |h_j|^2=\rho \quad\hbox{on }U,
$$
where, when $N=\infty$, the series converges locally uniformly on $U$.
Set
$$
H=
\begin{cases}
(h_1,\ldots,h_N)\colon U\to\mathbb C^N, & N<\infty,\\
(h_1,h_2,\ldots)\colon U\to\ell^2, & N=\infty.
\end{cases}
$$
Assume that there exists a constant $\mu>0$ such that $\rho^\mu$ admits
a sesqui-holomorphic extension to $M\times M$. That is, there exists a
function $K$ on $M\times M$, holomorphic in the first variable and
anti-holomorphic in the second variable, such that
$K(z,z)=\rho(z)^\mu$ on $M$.
Then the following statements hold.
\begin{enumerate}[label=(\arabic*)]
\item $H$ admits a semi-single-valued holomorphic extension to $M$.

\item Each $h_j$ admits a semi-single-valued
holomorphic extension to $M$, and $|h_j|^2$ admits a single-valued
real-analytic extension to $M$. These extensions still satisfy $\sum_{j=1}^N |h_j|^2=\rho$ on $M.$

\item Assume $N \geq 2.$ Then the map
$[H]\colon(M,\omega)\to (\mathbb P^{N-1},\omega_{FS})$ defined as follows is
a well-defined holomorphic isometric immersion:~for $q\in M$, choose any
analytic continuation $\widetilde H$ of $H$ to a neighborhood of $q$ and
set $[H](q)=[\widetilde H(q)]$.
\end{enumerate}
\end{theorem}

The proof of Theorem~\ref{Theorem 2 and 2+ in Ming's Notes} is greatly inspired by Mok's work \cite{Mo12}, especially his elegant idea of using the Cauchy--Schwarz inequality to study the monodromy of Bergman isometric maps. We first establish the following proposition.

\begin{pro}
\label{First half of Theorem 2 and 2+ in Ming's Notes}
Let $(M,\omega)$ be a K\"ahler manifold such that there exists an
everywhere positive real-analytic function $\rho$ on $M$ satisfying
$\omega=i\partial\overline{\partial}\log\rho$. Let $U$ be an open
connected subset of $M$, let $1 \leq N\leq\infty$, and let
$\{h_j\}_{j=1}^N$ be holomorphic functions on $U$ such that
\begin{equation}\label{Definitino of rho as a sum}
\sum_{j=1}^N |h_j|^2=\rho \quad\hbox{on }U,
\end{equation}
where, when $N=\infty$, the series converges locally uniformly on $U$.

\begin{enumerate}[label=(\arabic*)]
\item Let $\gamma_i\colon[0,1]\to M$, $i=1,2$, be continuous curves with
$\gamma_i(0)\in U$. Then, for every $t\in[0,1]$, there exist neighborhoods
$U_{t,i}$ of $\gamma_i(t)$ such that each $h_j$ admits an
analytic continuation along $\gamma_i|_{[0,t]}$ to a holomorphic function
$(h_j)_{t,i}$ on $U_{t,i}$. Moreover, the series
\begin{equation}\label{series continuation in both variables}
\sum_{j=1}^N (h_j)_{t,1}(z)\overline{(h_j)_{t,2}(w)},
\qquad (z,w)\in U_{t,1}\times U_{t,2},
\end{equation}
converges locally uniformly on $U_{t,1}\times U_{t,2}$ and defines a
function that is holomorphic in $z$ and anti-holomorphic in $w$. In
addition,
$$
\sum_{j=1}^N |(h_j)_{t,i}(z)|^2=\rho(z)
\quad\hbox{on }U_{t,i}, \qquad i=1,2.
$$

\item If $M$ is simply connected, then each $h_j$
extends to a global holomorphic function on $M$, which we continue to
denote by $h_j$. Moreover, $\sum_{j=1}^N |h_j|^2$ converges locally
uniformly on $M$ and
$$
\sum_{j=1}^N |h_j|^2=\rho \quad\hbox{on }M.
$$
\end{enumerate}
\end{pro}

\begin{proof}
When $N<\infty$, by appending a zero component to
$(h_1,\ldots,h_N)$, we may assume that $N\geq2$.
We prove parts~\textup{(1)} and~\textup{(2)} separately.

\medskip
\noindent{\bf Proof of (1):}
We first prove the analytic continuation statement along a single
continuous curve $\gamma\colon[0,1]\to M$ with $\gamma(0)\in U$.

For $0<t_0\leq1$, we say that $H=(h_1,\ldots,h_N)$, or
$H=(h_1,h_2,\ldots)$ when $N=\infty$, admits the desired analytic
continuation along $\gamma|_{[0,t_0]}$ if, for every $0\leq t\leq t_0$,
there exists a neighborhood $V_t$ of $\gamma(t)$ such that every $h_j$
admits analytic continuation along $\gamma|_{[0,t]}$ to a holomorphic
function $(h_j)_t$ on $V_t$, and
$$
\sum_{j=1}^N |(h_j)_t|^2=\rho \quad\hbox{on }V_t,
$$
where, when $N=\infty$, the series converges locally uniformly on $V_t$.
In particular, since $\rho>0$, the functions $(h_j)_t$ have no common
zeros on $V_t$. Let
$$
S=\{t_0\in[0,1]: H \hbox{ admits the desired analytic continuation
along } \gamma|_{[0,t_0]}\}
$$
and set $T=\sup S$. Since $\gamma(0)\in U$ and $U$ is open, $S$ contains
an interval $[0,\varepsilon]$ for some $\varepsilon>0$. Hence $T>0$.
By the definition of $T$, $H$ admits the desired analytic continuation
along $\gamma|_{[0,T)}$.

For $0\leq t<T$, write
$$
H_t=
\begin{cases}
((h_1)_t,\ldots,(h_N)_t), & N<\infty,\\
((h_1)_t,(h_2)_t,\ldots), & N=\infty.
\end{cases}
$$
Then $[H_t]$ is a holomorphic isometric immersion of $(V_t,\omega)$ into $(\mathbb P^{N-1},\omega_{FS})$.

On the other hand, by Calabi's extension theorem
(Theorem~\ref{Calabi Extension Theorem}), the initial projective map
$[H]$ admits an analytic continuation along $\gamma$ as a holomorphic
isometric immersion into $(\mathbb P^{N-1},\omega_{FS})$.
Let $[G_T]$ be such a continuation on a neighborhood $W_T$ of
$\gamma(T)$, which we may choose to be
simply connected. 
Choose a holomorphic representation
$$
G_T=
\begin{cases}
(g_1,\ldots,g_N), & N<\infty,\\
(g_1,g_2,\ldots), & N=\infty,
\end{cases}
$$
of $[G_T]$ on $W_T$. When $N=\infty$, the series
$\sum_{j=1}^{\infty}|g_j|^2$ converges locally uniformly on $W_T$.


Choose $t^*<T$ sufficiently close to $T$ so that $\gamma([t^*,T])\subseteq W_T$, and shrink $V_{t^*}$ if necessary so that $V_{t^*}\subseteq W_T$. Since $[H_{t^*}]$ and $[G_T]$ are obtained by analytic continuation of the same initial map $[H]$ along $\gamma$, uniqueness of analytic continuation gives
$$
[H_{t^*}]=[G_T]\quad\hbox{on }V_{t^*}.
$$
It follows that there exists a nowhere-vanishing holomorphic function
$\lambda$ on $V_{t^*}$ such that
$$
G_T=\lambda H_{t^*}\quad\hbox{on }V_{t^*}.
$$
Hence
\begin{equation}\label{eq:lambda-rho}
\sum_{j=1}^N|g_j|^2
=
|\lambda|^2\sum_{j=1}^N|(h_j)_{t^*}|^2
=
|\lambda|^2\rho
\quad\hbox{on }V_{t^*}.
\end{equation}
Since $[G_T]$ is a holomorphic isometric immersion,
$$
i\partial\overline{\partial}
\log\sum_{j=1}^N|g_j|^2
=
\omega
=
i\partial\overline{\partial}\log\rho
\quad\hbox{on }W_T.
$$
Thus
$$
\log\sum_{j=1}^N|g_j|^2-\log\rho
$$
is pluriharmonic on the simply connected domain $W_T$. Therefore, there
exists a nowhere-vanishing holomorphic function $\mu$ on $W_T$ such that
\begin{equation}\label{eq:mu-rho}
\sum_{j=1}^N|g_j|^2=|\mu|^2\rho
\quad\hbox{on }W_T.
\end{equation}
Comparing \eqref{eq:lambda-rho} and \eqref{eq:mu-rho}, we obtain
$|\lambda/\mu|\equiv1$ on $V_{t^*}$. Hence
$\lambda=e^{i\theta}\mu$ on $V_{t^*}$ for some $\theta\in\mathbb R$.
Replacing $\mu$ by $e^{i\theta}\mu$, we may assume that
$\lambda=\mu$ on $V_{t^*}$. It follows that
$$
H_{t^*}=\frac{1}{\mu}G_T\quad\hbox{on }V_{t^*}.
$$
Since the right-hand side is holomorphic on all of $W_T$, it gives an
analytic continuation of all of the components of $H$ to the common
neighborhood $W_T$ of $\gamma(T)$. Moreover, by
\eqref{eq:mu-rho},
$$
\sum_{j=1}^N\left|\frac{g_j}{\mu}\right|^2
=
\rho
\quad\hbox{on }W_T,
$$
with locally uniform convergence when $N=\infty$.
Thus the desired continuation of $H$ exists across $\gamma(T)$. If $T<1$, then
by continuity of $\gamma$, there exists some small $\delta>0$ such that
$\gamma([T,T+\delta])\subseteq W_T$. It follows that
$T+\delta\in S$,  contradicting
the definition of $T$. Therefore $T=1$. The same argument at $T=1$
gives the desired continuation at $\gamma(1)$ as well.

Applying this argument separately to $\gamma_1$ and $\gamma_2$, for every
$t\in[0,1]$ we obtain neighborhoods $U_{t,i}$ of $\gamma_i(t)$ and
analytic continuations $(h_j)_{t,i}$, $i=1,2$, satisfying
$$
\sum_{j=1}^N |(h_j)_{t,i}|^2=\rho
\quad\hbox{on }U_{t,i},
$$
with locally uniform convergence when $N=\infty$.
It remains to consider the series in
\eqref{series continuation in both variables}. When $N<\infty$, there is
nothing to prove. Suppose that $N=\infty$. Let
$K_i\Subset U_{t,i}$, $i=1,2$, be compact subsets. By the
Cauchy--Schwarz inequality,
$$
\sup_{(z,w)\in K_1\times K_2}
\left|
\sum_{j=m}^n
(h_j)_{t,1}(z)\overline{(h_j)_{t,2}(w)}
\right|
\leq
\left(
\sup_{z\in K_1}\sum_{j=m}^n |(h_j)_{t,1}(z)|^2
\right)^{1/2}
\left(
\sup_{w\in K_2}\sum_{j=m}^n |(h_j)_{t,2}(w)|^2
\right)^{1/2}.
$$
The right-hand side tends to zero as $m,n\to\infty$. Hence
\eqref{series continuation in both variables} converges locally uniformly
on $U_{t,1}\times U_{t,2}$. Its sum is holomorphic in $z$ and
anti-holomorphic in $w$. 
This proves part (1).

\medskip
\medskip
\noindent{\bf Proof of (2):}
By part~(1), each $h_j$ admits an analytic continuation along every curve in
$M$. Since $M$ is simply connected, the monodromy theorem implies that
each $h_j$ extends to a global holomorphic function on $M$. The identity
$
\sum_{j=1}^N |h_j|^2=\rho
$
and the local uniform convergence follow from part~(1). 
\end{proof}

We are ready to prove Theorem~\ref{Theorem 2 and 2+ in Ming's Notes}.


\begin{proof}[Proof of Theorem~\ref{Theorem 2 and 2+ in Ming's Notes}]
By part (1) of Proposition~\ref{First half of Theorem 2 and 2+ in Ming's Notes},
$H$ admits a (possibly multi-valued) holomorphic extension along every
continuous curve in $M$. We first show that this extension is
semi-single-valued.
Let $K$ be the function on $M \times M$ given in the assumption.
That is, $K$ is holomorphic in the first variable and anti-holomorphic in the second variable, and
$K(z,z)=\rho(z)^\mu$. Fix $p\in U$. After shrinking $U$ around $p$ if
necessary, we may assume that $U$ is simply connected and that $K$ does
not vanish on $U\times U$. Choose a branch of $K^{1/\mu}$ on
$U\times U$ such that $K(z,z)^{1/\mu}=\rho(z)$. By complexifying the
identity $\sum_{j=1}^N|h_j|^2=\rho$, we obtain
\begin{equation}\label{complexification-rho}
\sum_{j=1}^N h_j(z)\overline{h_j(w)}
=
K(z,w)^{1/\mu},
\qquad (z,w)\in U\times U.
\end{equation}
When $N=\infty$, the series on the left converges locally uniformly on
$U\times U$ by the Cauchy--Schwarz inequality.

Let $\gamma\colon[0,1]\to M$ be a loop based at $p$.
By part (1) of
Proposition~\ref{First half of Theorem 2 and 2+ in Ming's Notes}, each $h_j$ admits an analytic continuation along $\gamma$ to a germ of a holomorphic function
$\widehat h_j$ at $p$. Moreover,  after
shrinking $U$ if necessary, we can assume all of the $\widehat h_j$ are holomorphic
on $U$ and
\begin{equation}\label{norm-hat-h}
\sum_{j=1}^N|\widehat h_j|^2=\rho
\qquad\hbox{on }U,
\end{equation}
with locally uniform convergence when $N=\infty$.
Consider
$$
V=\{z\in M:K(z,p)=0\}.
$$
Since $K(p,p)=\rho(p)^\mu>0$, $V$ is a proper complex analytic
subvariety of $M$. After a small perturbation of $\gamma$, fixed near
its endpoints, we may assume that
$\gamma([0,1])\cap V=\varnothing$. Such a perturbation may be chosen
homotopic to the original loop relative to its endpoints and therefore
does not change the germs $\widehat h_j$ obtained by analytic
continuation.

Consider now the loop
$$
\alpha(t)=(\gamma(t),p), \qquad 0\leq t\leq1,
$$
in $M\times M$. Since $K$ does not vanish along $\alpha$ and
$\alpha([0,1])$ is compact, we can choose a neighborhood $\mathcal N$
of $\alpha([0,1])$ in $M\times M$ such that
$K(z,w)\neq0$ for all $(z,w)\in\mathcal N$. Consequently, the chosen branch $K^{1/\mu}$ in
\eqref{complexification-rho} admits continuation along $\alpha$ through
local branches on neighborhoods in $\mathcal N$, each holomorphic in
the first variable and anti-holomorphic in the second. Denote the resulting local branch near $(p,p)$ by
$\widetilde K^{1/\mu}$. Note $\widetilde K^{1/\mu}$ is again a local
branch of the $1/\mu$-power of $K(z,w)$ near $(p,p)$, holomorphic in
$z$ and anti-holomorphic in $w$.

By compactness of $\gamma([0,1])$, after shrinking $U$ around $p$ if
necessary, we may further assume that
$$
\gamma([0,1])\times U\subseteq\mathcal N.
$$
Fix $w\in U$. Regarding \eqref{complexification-rho} as an identity of
holomorphic functions in $z$, we continue both sides analytically along
$\gamma$. On the left, each $h_j$ is continued to $\widehat h_j$, while
$\overline{h_j(w)}$ remains fixed. On the right, the chosen branch of $K(\cdot,w)^{1/\mu}$ is continued along $\gamma$, yielding the corresponding branch of $\widetilde K^{1/\mu}(\cdot,w)$. By
part~\textup{(1)} of
Proposition~\ref{First half of Theorem 2 and 2+ in Ming's Notes}, the resulting series converges locally
uniformly. Since $w\in U$ was arbitrary, after shrinking $U$ if
necessary, we obtain
\begin{equation}\label{complexification-after-loop}
\sum_{j=1}^N \widehat h_j(z)\overline{h_j(w)}
=
\widetilde K^{1/\mu}(z,w),
\qquad (z,w)\in U\times U.
\end{equation}
Note that every local branch of the $1/\mu$-power of $K(z,w)$ has
modulus $|K(z,w)|^{1/\mu}$.
Setting $w=z$ in \eqref{complexification-rho} and
\eqref{complexification-after-loop}, and using
\eqref{norm-hat-h}, we therefore obtain on $U$
\begin{equation}\label{normequal-h}
\sum_{j=1}^N|h_j(z)|^2
=
\left|
\sum_{j=1}^N \widehat h_j(z)\overline{h_j(z)}
\right|
=
\rho(z)
=
\sum_{j=1}^N|\widehat h_j(z)|^2.
\end{equation}
Thus equality holds in the Cauchy--Schwarz inequality. Writing
$$
H=(h_1,\ldots,h_N),\qquad
\widehat H=(\widehat h_1,\ldots,\widehat h_N)
$$
when $N<\infty$, and using the analogous notation in $\ell^2$ when
$N=\infty$, we conclude that $H(z)$ and $\widehat H(z)$ are
proportional for every $z\in U$. Since $\rho>0$, $H$ is nowhere
vanishing. Hence there exists a holomorphic function $\lambda$ on $U$
such that
$
\widehat H=\lambda H.
$
The equality of the two norms in \eqref{normequal-h} gives
$|\lambda|\equiv1$. Therefore $\lambda$ is constant, and
$\widehat H=e^{i\theta}H$ for some $\theta\in\mathbb R$.
We have thus shown that an analytic continuation of $H$ around every loop
based at $p$ changes $H$ only by multiplication by a constant of modulus
one. It follows, by applying this to the loop obtained from any two
curves with the same initial and terminal points, that any two analytic
continuations $H_1$ and $H_2$ of $H$ to the same point satisfy
$
H_1=e^{i\theta}H_2
$
in a neighborhood of that point, for some $\theta\in\mathbb R$. This
proves part (1).

Part (2) follows immediately. Indeed, each component $h_j$ admits a
semi-single-valued holomorphic extension, and any two branches differ by
the same constant factor $e^{i\theta}$. Hence their squared moduli agree.
Thus the local functions $|h_j|^2$ obtained from the different branches
glue together to give a single-valued real-analytic function on $M$. It follows from the identity principle that $\sum_{j=1}^N |h_j|^2=\rho$ on $M.$

Finally, we prove part (3). By
Proposition~\ref{First half of Theorem 2 and 2+ in Ming's Notes}, 
any 
local branch $\widetilde H=(\widetilde h_1, \widetilde h_2, \cdots)$ of $H$ satisfies
$$
\sum_{j=1}^N|\widetilde h_j|^2=\rho.
$$
Since $\rho>0$, $\widetilde H$ has no common zero, and hence
$[\widetilde H]$ is well-defined. By part (1), two different branches of $H$
differ by multiplication by a constant $e^{i\theta}$, so they determine
the same projective map. Therefore
$[H]\colon M\to\mathbb P^{N-1}$
is a well-defined holomorphic map. Moreover, locally,
$$
i\partial\overline{\partial}
\log\sum_{j=1}^N|\widetilde h_j|^2
=
i\partial\overline{\partial}\log\rho
=
\omega.
$$
Hence $[H]\colon(M,\omega)\to(\mathbb P^{N-1},\omega_{FS})$ is a
holomorphic isometric immersion. This proves part (3) and completes the proof.
\end{proof}

As an application of Theorem~\ref{Theorem 2 and 2+ in Ming's Notes}, we
prove the following proposition. 

\begin{pro}\label{f over varphi is semi-single-valued}
Let $D\subseteq\Omega$ be domains in $\mathbb C^n$, and assume that $\Omega$ admits a Bergman metric. Suppose that there exist a constant $\lambda>0$, an open connected subset $O\subseteq D$, and a nowhere-vanishing holomorphic function $\varphi$ on $O$ such that
\begin{equation}\label{Metric equality up to a constant - 2}
K_D=|\varphi|^2K_\Omega^\lambda \quad\hbox{on }O.
\end{equation}
Then, for every $f\in A^2(D)$, the function
$\frac{f}{\varphi}$ admits a semi-single-valued holomorphic extension to
$\Omega$, and $\left|\frac{f}{\varphi}\right|^2$ admits a single-valued
real-analytic extension to $\Omega$. If $\Omega$ is simply connected,
then the holomorphic extension of $\frac{f}{\varphi}$ is single-valued.
\end{pro}

\begin{proof}
The conclusion is trivial if $f\equiv0$, so assume $f\not\equiv0$ (and hence the Bergman space of $D$ has dimension $N \geq 1$). After scaling, we may assume that $\|f\|_{L^2(D)}=1$. Choose an orthonormal basis $\{\phi_k\}_{k=1}^N$ of $A^2(D)$ with $\phi_1=f$, and set $\widetilde\phi_k:=\frac{\phi_k}{\varphi}$ on $O$. Since $K_D(z,z)=\sum_{k=1}^N|\phi_k(z)|^2$, it follows from \eqref{Metric equality up to a constant - 2} that
$$
\sum_{k=1}^N|\widetilde\phi_k(z)|^2=K_\Omega^\lambda(z,z)\quad\hbox{on }O.
$$
We apply Theorem~\ref{Theorem 2 and 2+ in Ming's Notes} to $(\Omega,\lambda\omega_\Omega)$ with $\rho=K_\Omega^\lambda$ and $\mu=\frac{1}{\lambda}$. Indeed, $\lambda\omega_\Omega=i\partial\overline{\partial}\log\rho$, while $\rho^\mu=K_\Omega$, and the Bergman kernel $K_\Omega(z,w)$ gives a globally defined extension of $\rho^\mu$ to $\Omega\times\Omega$ that is holomorphic in the first variable and anti-holomorphic in the second. Thus, by part~(2) of Theorem~\ref{Theorem 2 and 2+ in Ming's Notes}, $\widetilde\phi_1=\frac{f}{\varphi}$ admits a semi-single-valued holomorphic extension to $\Omega$, and $\left|\frac{f}{\varphi}\right|^2$ admits a single-valued real-analytic extension. If $\Omega$ is simply connected, then part~(2) of Proposition~\ref{First half of Theorem 2 and 2+ in Ming's Notes} shows that the extension of $\frac{f}{\varphi}$ is single-valued.
\end{proof}

\section{Proof of Theorem \ref{The main Stein result}}\label{Section Theorem 1 proved}

To prepare for the proof of Theorem~\ref{The main Stein result}, we first establish the following results.

\begin{lem}\label{Proposition that shows that Omega minus D is a closed pluripolar subset}
    Let $D \subseteq \Omega$ be domains in $\mathbb{C}^n$ such that $D$ is bounded and pseudoconvex, and $\Omega$ admits a Bergman metric.  If there is a $\lambda > 0$ such that
    \begin{equation}\label{Metric equality up to a constant}
    i\partial\overline{\partial}\log K_D = \lambda i\partial\overline{\partial}\log K_{\Omega} \quad\hbox{on }D,
    \end{equation}
    then $\Omega \setminus D$ is a closed pluripolar subset of $\Omega$. Consequently, $\lambda=1.$
\end{lem}
 \begin{proof}  By \eqref{Metric equality up to a constant}, there exists a ball $O \subseteq D$ and a nowhere-zero holomorphic function $\varphi$ on $O$ such that
$$
K_{D}(z, z) = |\varphi(z)|^2K_{\Omega}^{\lambda}(z, z) \quad \hbox{on $O$}.
$$
By applying  Proposition \ref{f over varphi is semi-single-valued} with $f= 1$, we have  $\eta(z) := {1 \over \varphi(z)}$ admits a semi-single-valued holomorphic extension to $\Omega$ and $|\eta(z)|^2$ is a real-analytic function on $\Omega$.  By analyticity,
\begin{equation}\label{Equation 29}
    |\eta(z)|^2K_{D}(z, z) = K_{\Omega}^{\lambda}(z, z) \quad \hbox{on }D.
\end{equation}
Let
$$
V_{\eta} = \{z \in \Omega: |\eta(z)|^2 = 0\}.
$$
For any sufficiently small ball $W\subseteq\Omega$, the set
$V_\eta\cap W$ is the zero locus of a local holomorphic branch of
$\eta$ on $W$. This implies that $\partial D \cap V_{\eta}$ is locally pluripolar. We  follow the ideas in \cite{HL24} and consider the disjoint union
$$
\partial D \cap \Omega = (\partial D \cap V_{\eta}) \cup ((\partial D \cap \Omega) \setminus V_{\eta}).
$$
Take any $p \in (\partial D \cap \Omega) \setminus V_{\eta}$.  Since $|\eta(p)|^2 \neq 0$, by \eqref{Equation 29},
$$
K_{D}(z, z) = {1 \over |\eta(z)|^2}K^{\lambda}_{\Omega}(z, z), \quad z \in D,\hbox{ sufficiently close to $p$}.
$$
It follows that $\limsup_{z \in D \to p} K_{D}(z, z) < \infty$.  By \cite[Lemma 11]{PZ02}, $(\partial D \cap \Omega) \setminus V_{\eta}$ is locally pluripolar.  Since locally pluripolar sets coincide with globally pluripolar sets, \cite{Jo78}, $\partial D \cap \Omega$ is pluripolar.  As a pluripolar set cannot separate interior and exterior points, $D$ has no exterior points in $\Omega$.   Thus, $\Omega\setminus D = \partial D \cap \Omega$ is pluripolar.  Consequently, $\lambda = 1$.
\end{proof}

\begin{remark}\label{Boundedness hypothesis used here remark - 1}
    The bounded pseudoconvex hypothesis on $D$ is needed in the application of \cite[Lemma 11]{PZ02} in the preceding proof.
\end{remark}

We next apply the new Calabi-type theorem developed in \S\ref{sec: new calabi thm} to prove the following proposition, which we believe is of interest in its own right.
It shows that, although no holomorphic isometric embedding of $(\Omega,\lambda\omega_\Omega)$ into
$(\mathbb{P}^{\infty}, \omega_{FS})$ is assumed a priori, such an embedding can nevertheless be obtained in our setting. It also justifies the assertion in Remark~\ref{rmk:non-projectively-induced-target}.

\begin{pro}\label{Lemma prior to the main Stein theorem} Let $M$ be an $n$-dimensional complex manifold with a well-defined Bergman metric $\omega_M$ and $\Omega \subseteq \mathbb{C}^n$ a bounded domain with Bergman metric $\omega_{\Omega}$.  Assume there exists an open connected set $U \subseteq M$ and a holomorphic map $f: U \to \Omega$ such that
$\omega_M = \lambda f^*(\omega_{\Omega}),$ for some positive $\lambda \in \mathbb{R}.$
Let $N$ with $2 \leq N \leq \infty$ denote the dimension of $A^2_{(n, 0)}(M)$ and define 
$$\omega_{\lambda} := \lambda \omega_{\Omega} = \lambda i\partial\overline{\partial}\log K_{\Omega}.$$  
Then there exists a holomorphic isometric embedding $[H]$ from $(\Omega, \omega_{\lambda})$ to $(\mathbb{P}^{N-1}, \omega_{FS})$.
\end{pro}

\begin{proof}
Since $\omega_M = \lambda f^*\omega_\Omega$ and both metrics are
positive definite, $df$ has full rank. Thus, after shrinking $U$, we may
assume that $(U,z)$ is a coordinate chart and that
$f\colon U\to V:=f(U)$ is biholomorphic. Let $g=f^{-1}\colon V\to U$.
Then
\begin{equation}\label{Equation 2++}
g^*\omega_M=\omega_\lambda.
\end{equation}
Let $\{\Phi_j\}_{j=1}^N$ be an orthonormal basis of
$A^2_{(n,0)}(M)$. Writing
$\Phi_j=\phi_j\,dz_1\wedge\cdots\wedge dz_n$ on $U$, we have
$$
\omega_M
=
i\partial\overline{\partial}
\log\sum_{j=1}^N|\phi_j|^2
\quad\hbox{on }U,
$$
where the series converges locally uniformly on $U$ when $N=\infty$. Set $\rho:=K_\Omega^\lambda$. By \eqref{Equation 2++},
$$
i\partial\overline{\partial}
\log\sum_{j=1}^N|\phi_j\circ g|^2
=i\partial\overline{\partial}\log\rho
\quad\hbox{on }V,
$$
where the series converges locally uniformly on $V$ when $N=\infty$.
After shrinking $V$, we may choose a nowhere-vanishing
holomorphic function $\varphi$ on $V$ such that
$$
\rho=|\varphi|^2\sum_{j=1}^N|\phi_j\circ g|^2.
$$
Set $h_j:=\varphi(\phi_j\circ g)$ and
$H:=(h_1,h_2,\ldots)$. When $N<\infty$, the latter means
$(h_1,\ldots,h_N)$. Then
\begin{equation}\label{in the new notation, definition of rho as a summation}
K_\Omega^\lambda=\rho=\sum_{j=1}^N|h_j|^2
\quad\hbox{on }V,
\end{equation}
where the series converges locally uniformly on $V$ when $N=\infty$.
We apply Theorem~\ref{Theorem 2 and 2+ in Ming's Notes} to
$(\Omega,\omega_\lambda)$ with $\rho=K_\Omega^\lambda$ and
$\mu=\frac{1}{\lambda}$. Indeed,
$\rho^\mu=K_\Omega$, whose polarization $K_\Omega(w,\eta)$ is globally
defined on $\Omega\times\Omega$. Hence part (3) of Theorem~\ref{Theorem 2 and 2+ in Ming's Notes} gives a holomorphic
isometric immersion
$$
[H]\colon(\Omega,\omega_\lambda)\longrightarrow
(\mathbb P^{N-1},\omega_{FS}).
$$

It remains to prove that $[H]$ is injective. Suppose otherwise, and
choose $q_1\neq q_2$ in $\Omega$ such that
$[H(q_1)]=[H(q_2)]$. Since $K_\Omega(\cdot,q_i)$ is not identically
zero, we may choose $p\in V$ such that
$K_\Omega(p,q_i)\neq0$ for $i=1,2$. After shrinking $V$ around $p$ if
necessary, we may assume that
\begin{equation}\label{Equation 4++}
K_\Omega(w,\eta)\neq0,
\qquad w,\eta\in V,
\end{equation}
and that $K_\Omega(w,q_i)\neq0$ for all $w\in V$ and $i=1,2$.
Consider
$$
Z:=\{z\in\Omega:K_\Omega(z,p)=0\}.
$$
Since $K_\Omega(p,p)>0$, $Z$ is a proper complex analytic subvariety
of $\Omega$. Moreover,
$K_\Omega(q_i,p)=\overline{K_\Omega(p,q_i)}\neq0$, so $q_i\notin Z$.
Since $\Omega\setminus Z$ is path connected, choose a path
$\alpha_i\colon[0,1]\to\Omega\setminus Z$ joining $p$ to $q_i$.

For $i=1,2$,
let $H_i=(h_{1,i},h_{2,i},\ldots)$ be the germ near $q_i$ obtained by
analytically continuing $H|_V$ along $\alpha_i$. By
part~\textup{(1)} of
Proposition~\ref{First half of Theorem 2 and 2+ in Ming's Notes}, we can choose a neighborhood $V_i$ of $q_i$ such that all of the
$h_{j,i}$ are holomorphic on $V_i$. Moreover, by
part~\textup{(3)} of
Theorem~\ref{Theorem 2 and 2+ in Ming's Notes},
$[H]=[H_i]$ on $V_i$. Since $[H(q_1)]=[H(q_2)]$, there exists
$c\neq0$ such that
\begin{equation}\label{Equation that says H1q1 equals cH2q2}
H_1(q_1)=cH_2(q_2).
\end{equation}
By \eqref{Equation 4++}, after shrinking $V$ if necessary,
complexifying
\eqref{in the new notation, definition of rho as a summation}
gives a branch of the $\lambda$-power of $K_\Omega(z,\eta)$ on
$V\times V$ such that
\begin{equation}\label{local complexification for injectivity}
K_\Omega^\lambda(z,\eta)
=
\sum_{j=1}^N h_j(z)\overline{h_j(\eta)},
\qquad (z,\eta)\in V\times V.
\end{equation}
We next continue this chosen branch along the paths $\alpha_i$. For
$i=1,2$, consider the path
$$
\beta_i(t)=(\alpha_i(t),p),\qquad 0\leq t\leq1,
$$
in $\Omega\times\Omega$. Since
$\alpha_i([0,1])\cap Z=\varnothing$, $K_\Omega$ does not vanish along
$\beta_i$. By compactness of $\alpha_i([0,1])$, we may choose a
neighborhood $\mathcal N_i$ of $\beta_i([0,1])$ in
$\Omega\times\Omega$ on which $K_\Omega$ does not vanish. After
shrinking $V$ around $p$ if necessary, we may further assume that
$$
\alpha_i([0,1])\times V\subseteq\mathcal N_i.
$$
Consequently, the chosen branch $K_\Omega^\lambda$ in
\eqref{local complexification for injectivity} admits continuation
along $\beta_i$ through local branches on neighborhoods in
$\mathcal N_i$, each holomorphic in the first variable and
anti-holomorphic in the second. Denote the resulting local branch near $(q_i,p)$ by
$\widetilde K_{\Omega,i}^\lambda$. After shrinking $V_i$ and $V$ if
necessary, $\widetilde K_{\Omega,i}^\lambda$ is a branch of the
$\lambda$-power of $K_\Omega(z,\eta)$ on $V_i\times V$.

On the other hand, fix $\eta\in V$ and
regard \eqref{local complexification for injectivity} as an identity of
holomorphic functions in $z$. Continuing both sides analytically along
$\alpha_i$, each $h_j$ continues to $h_{j,i}$, while
$\overline{h_j(\eta)}$ remains fixed. On the left, the chosen branch of
$K_\Omega(z,\eta)^\lambda$ continues to
$\widetilde K_{\Omega,i}^\lambda(z,\eta)$. By
part~\textup{(1)} of
Proposition~\ref{First half of Theorem 2 and 2+ in Ming's Notes}, the resulting series converges locally
uniformly. Since $\eta\in V$ was arbitrary, we obtain, for $i=1, 2,$
\begin{equation}\label{Equation 6++}
\widetilde K_{\Omega,i}^\lambda(z,\eta)
=
\sum_{j=1}^N h_{j,i}(z)\overline{h_j(\eta)},
\qquad (z,\eta)\in V_i\times V.
\end{equation}
Evaluating \eqref{Equation 6++} at $z=q_i$ for $i=1, 2$ respectively, and using
\eqref{Equation that says H1q1 equals cH2q2}, we obtain
$$
\widetilde K_{\Omega,1}^\lambda(q_1,\eta)
=
c\,\widetilde K_{\Omega,2}^\lambda(q_2,\eta),
\qquad \eta\in V.
$$
Every local branch of the $\lambda$-power of $K_\Omega$ has modulus
$|K_\Omega|^\lambda$. Hence
$$
|K_\Omega(q_1,\eta)|^\lambda
=
|c|\,|K_\Omega(q_2,\eta)|^\lambda,
\qquad \eta\in V.
$$
Set $a=|c|^{1/\lambda}>0$. By the Hermitian symmetry of the Bergman
kernel,
$$
|K_\Omega(\eta,q_1)|
=
a|K_\Omega(\eta,q_2)|,
\qquad \eta\in V.
$$
Since $K_\Omega(\eta,q_2)\neq0$ on $V$, the holomorphic function
$
\frac{K_\Omega(\eta,q_1)}
{aK_\Omega(\eta,q_2)}
$
has constant modulus one. Therefore there exists $\theta\in\mathbb R$
such that
$$
K_\Omega(\eta,q_1)
=
ae^{i\theta}K_\Omega(\eta,q_2)
\quad\hbox{on }V.
$$
By analyticity, this identity holds for every $ \eta \in\Omega$. Writing
$b=ae^{i\theta}$ and using the reproducing property, we obtain
$f(q_1)=\overline b\,f(q_2)$ for every $f\in A^2(\Omega)$. This is impossible as
$\Omega$ is bounded. Indeed, note $1,z_1,\ldots,z_n\in A^2(\Omega)$. Taking $f=1$
gives $b=1$, and then taking $f=z_k$, $1\leq k\leq n$, gives
$q_1=q_2$, a contradiction.
Thus $[H]$ is injective and hence is a holomorphic isometric embedding.
\end{proof}

    We now prove the first main theorem of the paper.

\begin{proof}[Proof of Theorem~\ref{The main Stein result}]
We first assume that $\lambda>0$ is constant. Set
$\omega_\lambda=\lambda\omega_\Omega$, so that
$f^*\omega_\lambda=\omega_M$. Let
$N=\dim A^2_{(n,0)}(M)$. Since $\omega_M$ is a well-defined Bergman
metric, 
we have $N>1$.
Let
$\beta_M\colon M\to\mathbb P^{N-1}$ be the Bergman--Bochner map of $M$. By Huang--Li
\cite[Proposition~3.1 and Remark~3.4]{HuLi26}, $A^2_{(n,0)}(M)$ separates points of $M$,
which is equivalent to $\beta_M$ being injective.
Then it follows from Mok's framework \cite{Mo12} that $f$ extends to a biholomorphism $F$ from $M$ onto a subdomain $F(M)$ of $\Omega$. For the reader's convenience and to keep the argument self-contained, we sketch a proof of this fact using Proposition~\ref{Lemma prior to the main Stein theorem} and Calabi's theorem.

Let
$[H]\colon(\Omega,\omega_\lambda)\to(\mathbb P^{N-1},\omega_{FS})$
be the holomorphic isometric embedding given by
Proposition~\ref{Lemma prior to the main Stein theorem}.
By Calabi's rigidity theorem (Theorem \ref{thm:Rigidity}), there exists a unitary isometry $L$ from
the smallest closed projective subspace containing the image of
$\beta_M|_U$ onto the smallest closed projective subspace containing the
image of $[H]|_{f(U)}$ such that
\begin{equation}\label{Equation 8++}
L\circ\beta_M=[H]\circ f\quad\hbox{on }U.
\end{equation}
Since $\omega_\Omega$ is complete, so is
$\omega_\lambda=\lambda\omega_\Omega$. Hence, by
\cite[Chapter VI, Theorem 6.3]{Ko63}, $f$ admits an analytic continuation
along every curve in $M$ to a possibly multi-valued local holomorphic
isometry $F$ with values in $\Omega$. Analytically continuing
\eqref{Equation 8++} along any curve, every local branch of $F$ satisfies
\begin{equation}\label{Equation 9+++}
L\circ\beta_M=[H]\circ F.
\end{equation}
Since the left-hand side is globally defined and single-valued, all
branches of $[H]\circ F$ agree. As $[H]$ is injective, any two local
branches of $F$ agree. Thus $F$ is single-valued and defines a global
 holomorphic map $F\colon M\to\Omega$ extending $f$.
Finally since both $L$ and $\beta_M$ are
injective, it follows from \eqref{Equation 9+++} that $F$ is injective. Setting $D:=F(M) \subseteq \Omega$, we conclude
that $F\colon M\to D$ is biholomorphic. 

The biholomorphic invariance of the Bergman
metric gives $F^*\omega_D=\omega_M$. On the other hand, analytic
continuation of the original local isometry gives
$F^*(\lambda\omega_\Omega)=\omega_M$. Therefore
$\omega_D=\lambda\omega_\Omega$ on $D$. Moreover, $D$ is biholomorphic
to the Stein manifold $M$, and hence is 
pseudoconvex. Lemma~\ref{Proposition that shows that Omega minus D is a closed pluripolar subset}
now applies and yields $\lambda=1$ and that $\Omega\setminus D$ is a
closed pluripolar subset of $\Omega$.

Finally, suppose that $n\geq2$ and
$\omega_M=\Lambda f^*\omega_\Omega$ on $U$, where $\Lambda>0$. By
Proposition~\ref{Conformal implies constant proposition in dimension greater than one},
$\Lambda$ is a positive constant on $U$. The conclusion then follows
from the first part of the proof.
\end{proof}

    \begin{remark}
        As $M$ is biholomorphic to a bounded domain $D$, necessarily $N = \infty$ in the preceding proof.
    \end{remark}

We finally give a brief proof for the statement in Remark \ref{rmk:question21dimension1}.

\begin{proof}[Proof of Remark \ref{rmk:question21dimension1}]
We only need to justify the case when $M$ is a Bergman-nondegenerate
Riemann surface with a well-defined Bergman metric, since the other two
cases follow immediately from Theorem~\ref{The main Stein result}. In this
case, the statement in Remark~\ref{rmk:question21dimension1} follows from the same proof as
Theorem~\ref{The main Stein result}, with a simple modification. Indeed,
at the beginning of that proof, we showed by using \cite{HuLi26} that
$A^2_{(n,0)}(M)$ separates points of $M$. In the present setting, this is
guaranteed directly by the Bergman-nondegeneracy assumption. Finally,
Lemma~\ref{Proposition that shows that Omega minus D is a closed pluripolar subset}
is still applicable, since every planar domain is pseudoconvex; in
particular, $F(M)\subseteq\Omega$ is pseudoconvex.
\end{proof}

\section{Preparation for the proof of Theorem~\ref{Main result - bounded homogoeneous case}}\label{Section that proves the bounded homogeneous case}

We first prove the following preliminary result. 

\begin{pro}\label{preliminary}
Let $D\subseteq\Omega$ be domains in $\mathbb C^n$, with $D$ bounded, and assume that $\Omega$ admits a Bergman metric. Suppose that
\begin{equation}\label{i ddbar log kD equals lambda i ddbar log kOmega - this has appeared before - 3}
i\partial\overline{\partial}\log K_D
=
\lambda i\partial\overline{\partial}\log K_\Omega
\quad\hbox{on }D
\end{equation}
for some constant $\lambda>0$. Then, for every $p\in D$, there exist a
neighborhood $O\subseteq D$ of $p$ and a nowhere-vanishing holomorphic
function $\varphi$ on $O$ with the following properties: $\frac{1}{\varphi}$ admits a semi-single-valued holomorphic
extension to $\Omega$, and
$\left|\frac{1}{\varphi}\right|^2$ admits a single-valued real-analytic
extension to $\Omega$. Moreover,
 $\varphi$ admits a semi-single-valued
holomorphic extension to $D,$  and this extension satisfies
\begin{equation}\label{First appearance of the functional equation of the Bergman kernel for the bounded homogeneous case}
K_D(z,z)
=
|\varphi(z)|^2K_\Omega^\lambda(z,z)
\quad\hbox{on }D.
\end{equation}
\end{pro}

\begin{proof}
Fix $p\in D$. By
\eqref{i ddbar log kD equals lambda i ddbar log kOmega - this has appeared before - 3},
the function
$
\log K_D-\lambda\log K_\Omega
$
is pluriharmonic on $D$. Hence, after shrinking to a simply connected
neighborhood $O\subseteq D$ of $p$, there exists a nowhere-vanishing
holomorphic function $\varphi$ on $O$ such that
\begin{equation}\label{local functional equation for Theorem I prime - 4}
K_D=|\varphi|^2K_\Omega^\lambda
\quad\hbox{on }O.
\end{equation}
Since $D$ is bounded, $1\in A^2(D)$. Applying
Proposition~\ref{f over varphi is semi-single-valued} to
\eqref{local functional equation for Theorem I prime - 4} with $f=1$, we
see that $\frac{1}{\varphi}$ admits a semi-single-valued holomorphic
extension to $\Omega$, while
$\left|\frac{1}{\varphi}\right|^2$ admits a single-valued real-analytic
extension to $\Omega$. Denote the latter extension by $u$. On $O$,
$$
u=\frac{K_\Omega^\lambda}{K_D}.
$$
Since both sides are real analytic on $D$, the identity principle for
real-analytic functions gives
\begin{equation}\label{u equals the quotient of Bergman kernels}
u=\frac{K_\Omega^\lambda}{K_D}>0
\quad\hbox{on }D.
\end{equation}
Thus every local branch of $\frac{1}{\varphi}$ is nowhere vanishing on
$D$. Taking reciprocals, we obtain a semi-single-valued holomorphic
extension of $\varphi$ to $D$. Moreover,
\eqref{u equals the quotient of Bergman kernels} gives
\begin{equation*}
K_D(z,z)=|\varphi(z)|^2K_\Omega^\lambda(z,z)
\quad\hbox{on }D.
\end{equation*}
Here $|\varphi|^2$ is single-valued, since any two branches of
$\varphi$ differ by a unimodular constant. This proves the desired statement.
\end{proof}

\subsection{Sufficient condition for $\lambda = 1$}

As mentioned in $\S$\ref{sec:introduction}, a major difficulty in the proof of
Theorem~\ref{Main result - bounded homogoeneous case} is to show that
$\lambda=1$. Accordingly, in this section we establish a sufficient
condition under which the normalizing constant $\lambda$ must equal one.
Recall that $\hbox{Aut}(\Omega)$ denotes the holomorphic automorphism group of $\Omega$.

\begin{pro}\label{The lemma that says in the BH case that lambda equals 1}
Let $D\subseteq\Omega\subseteq\mathbb C^n$ be bounded domains such that
\begin{equation}\label{The key identity between the Bergman metrics}
i\partial\overline{\partial}\log K_D
=
\lambda i\partial\overline{\partial}\log K_\Omega
\quad\hbox{on }D
\end{equation}
for some constant $\lambda>0$. 
Let $\mathcal A(\Omega)$ be a subgroup of $\hbox{Aut}(\Omega)$.  Define
$$
\widehat D:=\bigcup_{G\in \mathcal A(\Omega)}G^{-1}(D).
$$
Then $\widehat D$ is an $\mathcal A(\Omega)$-invariant domain
containing $D$, and $\widehat D\setminus D$ is Bergman-negligible in
$\widehat D$. In particular,
\begin{equation}\label{The key identity between the Bergman metrics is also satisfied for hatD}
K_D=K_{\widehat D}\quad\hbox{on }D,
\qquad
i\partial\overline{\partial}\log K_{\widehat D}
=
\lambda i\partial\overline{\partial}\log K_\Omega
\quad\hbox{on }\widehat D.
\end{equation}
Moreover, the following statements hold:
\begin{enumerate}[label=(\arabic*)]
\item Let $\Gamma_p$ denote the $\mathcal A(\Omega)$-orbit through $p$.  If there exists an analytically Zariski dense subset
$X\subseteq\Omega$ such that, for every $p\in X$, $\Gamma_p$
contains a nonconstant holomorphic curve, then $\lambda=1$.

\item If there exists $q\in D$ such that $\Gamma_q$ contains a
nonconstant holomorphic curve, then $\lambda=1$.
\end{enumerate}
\end{pro}

\begin{proof}
Fix $p\in D$. By Proposition \ref{preliminary}, we can choose a simply connected ball
$O\subseteq D$ centered at $p$ and a nowhere-vanishing
holomorphic function $\varphi$ on $O$ such that
\begin{equation}\label{The functional Bergman equation with varphi globally defined on all of D}
K_D=|\varphi|^2K_\Omega^\lambda
\quad\hbox{on }O.
\end{equation}
 Moreover, 
$\left|\frac{1}{\varphi}\right|^2$ admits a single-valued real-analytic
extension to $\Omega$; denote it by $u$. As in \eqref{u equals the quotient of Bergman kernels}, this extension satisfies
\begin{equation}\label{u times KD equals KOmega lambda}
uK_D=K_\Omega^\lambda\quad\hbox{on }D.
\end{equation}
In particular, $u>0$ on $D$.
Choose an orthonormal basis $\{\phi_k\}_{k=1}^\infty$ of $A^2(D)$ such
that $\phi_1$ is a nonzero constant. Then by \eqref{The functional Bergman equation with varphi globally defined on all of D},
$$
\sum_{k=1}^\infty\left|\frac{\phi_k}{\varphi}\right|^2
=K_\Omega^\lambda\quad\hbox{on }O.
$$
Set
$$
H=\left(\frac{\phi_1}{\varphi},
\frac{\phi_2}{\varphi},\ldots\right) \quad\hbox{on }O.
$$
Applying Theorem~\ref{Theorem 2 and 2+ in Ming's Notes} to
$(\Omega,\lambda\omega_\Omega)$ with $\rho=K_\Omega^\lambda$ and
$\mu=\frac{1}{\lambda}$, we see that $H,$ initially defined on $O,$ admits a semi-single-valued
holomorphic extension to $\Omega$.

\medskip
\noindent{\bf Step 1. Comparison of $D$ and its automorphic translates.}
Fix $G\in \mathcal A(\Omega)$ and set
$D_G:=G^{-1}(D)$ and $O_G:=G^{-1}(O)$. Since $O_G$ is simply connected
and $\det JG$ is nowhere vanishing, we may choose a branch of
$(\det JG)^{1-\lambda}$ on $O_G$ and set
\begin{equation}\label{deffinition of hat phi sub G}
\widehat\varphi_G
:=
(\varphi\circ G)(\det JG)^{1-\lambda} \quad\hbox{on }O_G.
\end{equation}
We first note that
\begin{equation}\label{Equation 14}
K_{D_G}=|\widehat\varphi_G|^2K_\Omega^\lambda
\quad\hbox{on }O_G.
\end{equation}
Indeed, letting $z\in O_G$, the transformation laws for the Bergman kernels
of $D$ and $\Omega$ give
$$
K_{D_G}(z,z)=|\det JG(z)|^2K_D(G(z),G(z)),
\qquad
K_\Omega(G(z),G(z))
=|\det JG(z)|^{-2}K_\Omega(z,z).
$$
Bringing $G(z) \in O$ into \eqref{The functional Bergman equation with varphi globally defined on all of D} yields,
$$K_D(G(z),G(z))=\big|\varphi(G(z))\big|^2 K_\Omega^\lambda (G(z),G(z)).$$
Combining the last three identities,
we obtain
$$
K_{D_G}(z,z)
=
\big|\varphi(G(z))\big|^2|\det JG(z)|^{2(1-\lambda)}
K_\Omega^\lambda(z,z)
=
|\widehat\varphi_G(z)|^2K_\Omega^\lambda(z,z),
$$
which proves \eqref{Equation 14}.
Next choose a curve $\gamma$ in $\Omega$ joining $p$ to $G^{-1}(p)$, and let
$$
A_G=(a_{1,G},a_{2,G},\ldots)
$$
be the branch of $H$ obtained by analytic continuation along $\gamma$
to $O_G$. 
By Theorem~\ref{Theorem 2 and 2+ in Ming's Notes},
\begin{equation}\label{AG norm equality}
\sum_{k=1}^\infty|a_{k,G}|^2=K_\Omega^\lambda
\quad\hbox{on }O_G.
\end{equation}
Here the series on the left  converges locally uniformly on $O_G$.
Choose an orthonormal basis $\{h_k\}_{k=1}^\infty$ of $A^2(D_G)$,
again with $h_1$ a nonzero constant, and set
$$
B_G=
\left(
\frac{h_1}{\widehat\varphi_G},
\frac{h_2}{\widehat\varphi_G},\ldots
\right).
$$
By \eqref{Equation 14},
\begin{equation}\label{BG norm equality}
\sum_{k=1}^\infty
\left|\frac{h_k}{\widehat\varphi_G}\right|^2
=K_\Omega^\lambda
\quad\hbox{on }O_G.
\end{equation}
Thus $[A_G]$ and $[B_G]$ are holomorphic isometric immersions of
$(O_G,\lambda\omega_\Omega)$ into
$(\mathbb P^\infty,\omega_{FS})$.

We claim that neither $[A_G](O_G)$ nor $[B_G](O_G)$ is contained in a
proper closed projective subspace of $\mathbb P^\infty$. Indeed, suppose
that $[A_G](O_G)$ were contained in such a subspace. Then there would
exist $(v_k)_{k=1}^\infty\in\ell^2$, not identically zero, such that
$$
\sum_{k=1}^\infty v_k a_{k,G}=0\quad\hbox{on }O_G.
$$
The series converges locally uniformly by the Cauchy--Schwarz inequality.
Analytically continuing this identity back along $\gamma$ gives
$$
\sum_{k=1}^\infty v_k\frac{\phi_k}{\varphi}=0
\quad\hbox{on }O,
$$
and hence $\sum_{k=1}^\infty v_k\phi_k=0$ on $D$. This contradicts the
orthonormality of $\{\phi_k\}$. A similar argument, now using the
orthonormal basis $\{h_k\}$ of $A^2(D_G)$, shows that $[B_G](O_G)$ is
also not contained in a proper closed projective subspace.

Calabi's rigidity theorem therefore gives a unitary operator
$\mathcal V_G\colon\ell^2\to\ell^2$ such that, after possibly multiplying
$\mathcal V_G$ by a unimodular constant,
$$
A_G=B_G\mathcal V_G\quad\hbox{on }O_G.
$$
Indeed, projective equality initially gives
$A_G=a_GB_G\mathcal V_G$ for a nowhere-vanishing holomorphic function
$a_G$ on $O_G$; comparing \eqref{AG norm equality} and
\eqref{BG norm equality} gives $|a_G|\equiv1$, so $a_G$ is a constant
of modulus one and can be absorbed into $\mathcal V_G$. Writing
$$
(\widetilde h_1,\widetilde h_2,\ldots)
:=(h_1,h_2,\ldots)\mathcal V_G,
$$
we obtain an orthonormal basis
$\{\widetilde h_k\}_{k=1}^\infty$ of $A^2(D_G)$ satisfying
\begin{equation}\label{Equation 34}
\widetilde h_k=\widehat\varphi_G a_{k,G}
\quad\hbox{on }O_G,\qquad k\geq1.
\end{equation}
We now apply a similar argument as in
\cite[Proof of Theorem~1.3]{BhGaNaXi26}. Since
$\{\phi_k\}_{k=1}^\infty$ and
$\{\widetilde h_k\}_{k=1}^\infty$ are orthonormal bases of
$A^2(D)$ and $A^2(D_G)$, respectively, there exists a unique unitary
operator
$$
\mathcal U_G\colon A^2(D)\longrightarrow A^2(D_G)
$$
such that
$
\mathcal U_G(\phi_k)=\widetilde h_k$, $k\geq1.
$
More precisely, let
$f=\sum_{k=1}^\infty c_k\phi_k\in A^2(D)$, where
$(c_k)_{k=1}^\infty\in\ell^2$. Then
\begin{equation}\label{UG expansion}
\mathcal U_G(f)=\sum_{k=1}^\infty c_k\widetilde h_k
\quad\hbox{in }A^2(D_G).
\end{equation}
The series on the right also converges locally uniformly on $D_G$.
On the other hand, letting $\gamma$ be chosen as above, by
Proposition~\ref{f over varphi is semi-single-valued},
$\frac{f}{\varphi}$ admits an analytic continuation along $\gamma$. Let
$\left(\frac{f}{\varphi}\right)_G$ denote the branch obtained by
analytic continuation along $\gamma$ from $O$ to $O_G$. We claim that
\begin{equation}\label{continuation of f over varphi along G}
\left(\frac{f}{\varphi}\right)_G
=
\sum_{k=1}^\infty c_k a_{k,G}
\quad\hbox{on }O_G.
\end{equation}
Indeed, for each $t\in[0,1]$, let $A_t:=(a_{k,t})_{k=1}^\infty$ denote the corresponding local branch of the analytic continuation of $H=\left(\frac{\phi_k}{\varphi}\right)_{k=1}^\infty$ on a common neighborhood $U_t$ of $\gamma(t)$, choosing $U_0=O$ and $U_1=O_G$, so that $A_0=H$ and $A_1=A_G$.
By
Theorem~\ref{Theorem 2 and 2+ in Ming's Notes},
$\sum_{k=1}^\infty |a_{k,t}|^2$ converges locally uniformly on $U_t$.
Since $(c_k)\in\ell^2$, the Cauchy--Schwarz inequality implies that
$\sum_{k=1}^\infty c_k a_{k,t}$ also converges locally uniformly on
$U_t$. These local sums are compatible on overlapping continuation
neighborhoods and, on the initial neighborhood $O$, equal
$$
\sum_{k=1}^\infty c_k\frac{\phi_k}{\varphi}
=
\frac{f}{\varphi}.
$$
They therefore give the analytic continuation of
$\frac{f}{\varphi}$ along $\gamma$. Since $A_1=A_G$, at the endpoint $\gamma(1),$ this gives
\eqref{continuation of f over varphi along G}.
Combining it with \eqref{Equation 34} and
\eqref{UG expansion}, we obtain
\begin{equation}\label{UG analytic continuation}
\mathcal U_G(f)
=
\widehat\varphi_G
\left(\frac{f}{\varphi}\right)_G
\quad\hbox{on }O_G.
\end{equation}
Set $h_G:=\mathcal U_G(1) \in A^2(D_G)$. Then
\begin{equation}\label{hG definition local}
h_G
=
\widehat\varphi_G
\left(\frac{1}{\varphi}\right)_G
\quad\hbox{on }O_G,
\end{equation}
and $h_G\not\equiv0$. Since every monomial $z^\alpha$ belongs to
$A^2(D)$ and is globally holomorphic on $\Omega$, its analytic
continuation along $\gamma$ is itself. Thus
\eqref{UG analytic continuation} and
\eqref{hG definition local} imply
$$
\mathcal U_G(z^\alpha)=\widehat\varphi_G
\left(\frac{z^\alpha}{\varphi}\right)_G=\widehat\varphi_G z^\alpha 
\left(\frac{1}{\varphi}\right)_G
=z^\alpha h_G
\quad\hbox{on }O_G,
$$
and hence, by the identity theorem,
\begin{equation}\label{UG monomial}
\mathcal U_G(z^\alpha)=z^\alpha h_G
\quad\hbox{on }D_G.
\end{equation}
Therefore, for all multi-indices $\alpha,\beta\in\mathbb N^n$,
\begin{equation}\label{moment equality D and DG}
\int_D z^\alpha\bar{z}^{\beta}\,dV
=
\int_{D_G}z^\alpha\bar{z}^{\beta}|h_G|^2\,dV.
\end{equation}
Note $D, D_G \subseteq \Omega$ are bounded domains. By uniqueness of compactly supported moment measures,
\begin{equation}\label{weighted measure identity}
\chi_D\,dV
=
|h_G|^2\chi_{D_G}\,dV.
\end{equation}
It follows immediately that $m(D\setminus D_G)=0$. Moreover,
\eqref{weighted measure identity} implies that $h_G=0$ almost everywhere
on $D_G\setminus D$. Since $h_G\not\equiv0$, its zero set has Lebesgue
measure zero, and hence $m(D_G\setminus D)=0$. Thus
\begin{equation}\label{DG symmetric difference}
m(D\triangle D_G)=0.
\end{equation}
Here $D\triangle D_G$ denotes their symmetric difference.
In particular, $D\cap D_G\neq\emptyset$. On this intersection,
\eqref{weighted measure identity} gives $|h_G|=1$ almost everywhere,
and hence everywhere by continuity. Since $D\cap D_G$ is a nonempty
open set, the open mapping theorem and
the identity theorem then give
\begin{equation}\label{hG unimodular constant}
h_G\equiv e^{i\theta_G}\quad\hbox{on }D_G
\end{equation}
for some $\theta_G\in\mathbb R$.
Applying Proposition~\ref{f over varphi is semi-single-valued} to
\eqref{Equation 14} with $f=1$, we see that
$\left|\frac{1}{\widehat\varphi_G}\right|^2$ admits a single-valued
real-analytic extension to $\Omega$. Denote this extension by $u_G$. 
By \eqref{hG definition local},
\eqref{hG unimodular constant} and the definition of $u$
$$
u_G
=
\left|
\left(\frac{1}{\varphi}\right)_G
\right|^2
=u
\quad\hbox{on }O_G.
$$
Thus, by real-analytic continuation,
\begin{equation}\label{uG equals u}
u_G=u\quad\hbox{on }\Omega.
\end{equation}
Furthermore, \eqref{Equation 14} gives
$u_GK_{D_G}=K_\Omega^\lambda$ on $O_G$, and hence, again by
real-analytic continuation,
\begin{equation}\label{u times KDG}
u_GK_{D_G}=uK_{D_G}=K_\Omega^\lambda
\quad\hbox{on }D_G.
\end{equation}
In particular, $u>0$ throughout $D_G$.

\medskip
\noindent{\bf Step 2. Construction of $\widehat D$.}
Set
$$
\widehat D
=
\bigcup_{G\in \mathcal A(\Omega)}D_G.
$$
Clearly $\widehat D$ is open and contains $D$. By
\eqref{DG symmetric difference}, every $D_G$ intersects $D$, so
$\widehat D$ is connected and hence is a domain. Moreover, for every
$F\in \mathcal A(\Omega) \subseteq \hbox{Aut}(\Omega)$,
$$
F^{-1}(\widehat D)
=
\bigcup_{G\in \mathcal A(\Omega)}(GF)^{-1}(D)
=
\widehat D,
$$
so $\widehat D$ is $\mathcal A(\Omega)$-invariant.
Since $\widehat D$ is second countable, the open cover
$\{D_G:G\in \mathcal A(\Omega)\}$ admits a countable subcover,
say $\widehat D=\bigcup_{j=1}^\infty D_{G_j}$. Hence, using \eqref{DG symmetric difference},
$$
m(\widehat D\setminus D)
\leq
\sum_{j=1}^\infty m(D_{G_j}\setminus D)=0.
$$
We next show that $\widehat D\setminus D$ is Bergman-negligible in
$\widehat D$. By \eqref{u times KDG}, $u>0$ on every $D_G$, and hence
on $\widehat D$. For $z\in\widehat D$, let
$(H_{1,z},H_{2,z},\ldots)$ be any local branch of the
semi-single-valued extension of $H$ at $z$. Since $\phi_1$ is a
nonzero constant, by the definitions of $H$ and $u,$
$$
|H_{1,z}|^2=|\phi_1|^2u>0.
$$
We may therefore define locally, for each $k\geq1$,
\begin{equation}\label{extension of phi k to hatD}
\widehat\phi_k
:=
\phi_1\frac{H_{k,z}}{H_{1,z}}.
\end{equation}
This is independent of the chosen branch, since any two branches of
$H$ differ by multiplication by the same unimodular constant. Hence
$\widehat\phi_k$ is a globally defined holomorphic function on
$\widehat D$. On $O$ it agrees with $\phi_k$, and therefore
$\widehat\phi_k=\phi_k$ on $D$ by the identity theorem.

Since $m(\widehat D\setminus D)=0$, the family
$\{\widehat\phi_k\}_{k=1}^\infty$ is orthonormal in
$A^2(\widehat D)$. It is also complete. Indeed, if
$f\in A^2(\widehat D)$ is orthogonal to every $\widehat\phi_k$, then
$f|_D$ is orthogonal to every $\phi_k$, so $f=0$ on $D$ and hence on
$\widehat D$. Thus $\{\widehat\phi_k\}_{k=1}^{\infty}$ is an orthonormal basis of
$A^2(\widehat D)$. The restriction map
$A^2(\widehat D)\to A^2(D)$ therefore sends an orthonormal basis onto
an orthonormal basis and is a unitary isomorphism. Consequently,
$\widehat D\setminus D$ is Bergman-negligible in $\widehat D$, and
$K_{\widehat D}=K_D$ on $D$. Since
\eqref{The key identity between the Bergman metrics} holds on $D$,
real-analytic continuation on the connected domain $\widehat D$ gives \eqref{The key identity between the Bergman metrics is also satisfied for hatD}.

\medskip
\noindent{\bf Step 3. An $\mathcal A(\Omega)$-invariant function.}
By \eqref{deffinition of hat phi sub G} and the definitions of $u$ and $u_G$, on $O_G$ we have
$$
u_G(z)
=
u(G(z))|\det JG(z)|^{-2(1-\lambda)}.
$$
Using \eqref{uG equals u}, this becomes
$$
u(z)
=
u(G(z))|\det JG(z)|^{-2(1-\lambda)}
\quad\hbox{on }O_G.
$$
Together with
$K_\Omega(G(z),G(z))
=|\det JG(z)|^{-2}K_\Omega(z,z)$, we obtain
$$
K_\Omega^{1-\lambda}(G(z),G(z))u(G(z))
=
K_\Omega^{1-\lambda}(z,z)u(z)
\quad\hbox{on }O_G.
$$
Both sides are real analytic on $\Omega$, so
\begin{equation}\label{orbit invariant function}
K_\Omega^{1-\lambda}(G(z),G(z))u(G(z))
=
K_\Omega^{1-\lambda}(z,z)u(z),
\qquad z\in\Omega,
\end{equation}
for every $G\in \mathcal A(\Omega)$. Thus
$K_\Omega^{1-\lambda}u$ is constant on every orbit $\Gamma_p$.

\medskip
\noindent{\bf Step 4. Proof of \textup{(1)} and \textup{(2)}.}
We first prove \textup{(1)}. Let $X \subseteq \Omega$ be the analytically
Zariski dense subset as in (1),  and set
$$
\mathcal Z:=\{z\in\Omega:u(z)=0\}.
$$
Locally, $u$ is the squared modulus of a holomorphic branch of
$\frac{1}{\varphi}$, so $\mathcal Z$ is a complex-analytic subset of
$\Omega$. It is proper since $u>0$ on $D$. As $X$ is analytically
Zariski dense, there exists $p_0\in X\setminus\mathcal Z$. By the conclusion in Step 3, there exists some constant $c_0$, such that $K_\Omega^{1-\lambda}u \equiv c_0$ on $\Gamma_{p_0}.$ We must have $c_0>0$, because $u(p_0)>0$.
By
assumption, $\Gamma_{p_0}$ contains the image of a nonconstant
holomorphic map $\eta\colon\mathbb D\to\Omega$.
This yields,
$$
K_\Omega^{1-\lambda}(\eta(\xi),\eta(\xi))
u(\eta(\xi))
\equiv c_0.
$$
Hence $u>0$ along the image of
$\eta$. Wherever $u>0$, locally $u=|g|^2$ for a nowhere-vanishing
holomorphic function $g$, so $\log u$ is pluriharmonic. Taking the logarithm of the preceding equation and then applying
$i\partial\overline{\partial}$, we obtain
$$
(1-\lambda)\eta^*\omega_\Omega=0.
$$
Since $\eta$ is nonconstant and $\omega_\Omega$ is positive definite,
$\eta^*\omega_\Omega$ is nonzero at some point. Therefore $\lambda=1$.

Finally, let $q\in D$ be as in \textup{(2)}. By
\eqref{u times KD equals KOmega lambda}, $u(q)>0$. Since $\Gamma_q$
contains a nonconstant holomorphic curve, the same argument as above applied to
$\Gamma_q$ gives $\lambda=1$. This completes the proof.
\end{proof}

\subsection{Rigidity when $\lambda=1$}

To conclude this section, we prove a rigidity result when the normalizing constant satisfies $\lambda=1$.

\begin{pro}\label{Theorem I'}
Let $D\subseteq\Omega$ be bounded domains in $\mathbb C^n$ such that
\begin{equation}\label{i ddbar log kD equals lambda i ddbar log kOmega - this has appeared before}
i\partial\overline{\partial}\log K_D
= i\partial\overline{\partial}\log K_\Omega
\quad\hbox{on }D.
\end{equation}
Then it holds that
\begin{equation}\label{First appearance of the functional equation of the Bergman kernel for lambda equal one}
K_D(z,z)
=K_\Omega(z,z)
\quad\hbox{on }D.
\end{equation}
Moreover, $\Omega\setminus D$ is
Bergman-negligible in $\Omega$.
\end{pro}

\begin{proof}
Fix $p\in D$. 
By Proposition \ref{preliminary}, we can choose a small simply connected
neighborhood $O\subseteq D$ of $p$ and a nowhere-vanishing
holomorphic function $\varphi$ on $O$ such that
\begin{equation}\label{local functional equation for Theorem I prime}
K_D=|\varphi|^2K_\Omega
\quad\hbox{on }O.
\end{equation}
Moreover, 
$\varphi$ admits a semi-single-valued holomorphic
extension to $D$. 
We now prove a stronger statement.
\begin{lem}\label{lemma varphi}
The function $\varphi$ extends to a nowhere-vanishing single-valued
holomorphic function on $D$, which we continue to denote by $\varphi$.  The extension satisfies
\begin{equation}\label{Equation *}
K_D(z,w)
=
\varphi(z)\overline{\varphi(w)}K_\Omega(z,w),
\qquad (z,w)\in D\times D.
\end{equation}
\end{lem}

\begin{proof}[Proof of Lemma \ref{lemma varphi}]
Polarizing
\eqref{local functional equation for Theorem I prime}, we obtain
\begin{equation}\label{local polarized kernel identity lambda one}
K_D(z,w)
=
\varphi(z)\overline{\varphi(w)}K_\Omega(z,w),
\qquad z,w\in O.
\end{equation}
Since $K_\Omega(p,p)>0$, after shrinking $O$ if necessary we may assume
that $K_\Omega(z,p)\neq0$ for every $z\in O$. Setting $w=p$ in
\eqref{local polarized kernel identity lambda one} gives
$$
\varphi(z)
=
\frac{K_D(z,p)}
{\overline{\varphi(p)}K_\Omega(z,p)}
\quad\hbox{on }O.
$$
Consequently,
\begin{equation}\label{meromorphic extension of varphi}
\Phi(z)
:=
\frac{K_D(z,p)}
{\overline{\varphi(p)}K_\Omega(z,p)}
\end{equation}
defines a single-valued meromorphic continuation of $\varphi$ to $D$.
Let $P\subseteq D$ denote the polar set of $\Phi$ in $D$. 
Since $\Phi=\varphi$ on $O$,  \eqref{local functional equation for Theorem I prime} gives
$$
|\Phi|^2=\frac{K_D}{K_\Omega}
\quad\hbox{on }O.
$$
Both sides are real analytic on $D\setminus P$. Since $P$ is a proper
complex-analytic subset of $D$, the set $D\setminus P$ is connected.
Hence, by the identity principle for real-analytic functions,
\begin{equation}\label{modulus of Phi}
|\Phi(z)|^2
=
\frac{K_D(z,z)}{K_\Omega(z,z)}
\quad\hbox{on }D\setminus P.
\end{equation}
The right-hand side is positive and real analytic on $D$, and hence
locally bounded. Therefore \eqref{modulus of Phi} shows that $\Phi$ is
locally bounded on $D\setminus P$ near every point of $P$. By the
Riemann extension theorem, $\Phi$ extends holomorphically across $P$. Equation~\eqref{modulus of Phi} also shows that the resulting
holomorphic function has no zeros. Thus $\varphi$ extends to a
nowhere-vanishing single-valued holomorphic function on $D$, which we
continue to denote by $\varphi$. Finally, by the identity principle again, 
\eqref{local polarized kernel identity lambda one} yields \eqref{Equation *}.
\renewcommand{\qedsymbol}{$\blacksquare$}
    \end{proof}
    \renewcommand{\qedsymbol}{$\square$}

We next prove $\varphi$ is necessarily a unimodular constant. For that, we first use a similar argument as in
\cite[Proposition~4.1]{BhGaNaXi26} to compare the Bergman spaces of
$\Omega$ and $D$. For $w\in D$, write
$$
K_w^\Omega(z):=K_\Omega(z,w),
\qquad
K_w^D(z):=K_D(z,w).
$$
Let
$$
\mathcal E_\Omega
:=
\operatorname{Span}\{K_w^\Omega:w\in D\}
\subseteq A^2(\Omega),
$$
where the span consists of finite linear combinations. Define
$T\colon\mathcal E_\Omega\to A^2(D)$ by
\begin{equation}\label{definition of T between Bergman spaces}
T(K_w^\Omega)
:=
\frac{1}{\overline{\varphi(w)}}K_w^D,
\qquad w\in D,
\end{equation}
and extend $T$ linearly.
We verify that $T$ is well defined and isometric. Indeed, for arbitrary
$w_1,\ldots,w_m\in D$ and $c_1,\ldots,c_m\in\mathbb C$,
\eqref{Equation *} and the reproducing property give

$$
\left\|\sum_{j=1}^m c_j\frac{K_{w_j}^D}{\overline{\varphi(w_j)}}\right\|_{A^2(D)}^2
=\sum_{j,k=1}^m c_j\overline{c_k}\frac{K_D(w_k,w_j)}{\overline{\varphi(w_j)}\varphi(w_k)}
=\sum_{j,k=1}^m c_j\overline{c_k}K_\Omega(w_k,w_j)
=\left\|\sum_{j=1}^m c_jK_{w_j}^\Omega\right\|_{A^2(\Omega)}^2.
$$
Thus $T$ is well defined and isometric.

Moreover, $\mathcal E_\Omega$ is dense in $A^2(\Omega)$. Indeed, if
$f\in A^2(\Omega)$ is orthogonal to every $K_w^\Omega$, $w\in D$, then
$f(w)=0$ for every $w\in D$, and hence $f\equiv0$ on $\Omega$ by the
identity theorem. Similarly, the image of $T$ contains
$$
\left\{
\frac{1}{\overline{\varphi(w)}}K_w^D:w\in D
\right\},
$$
whose linear span is dense in $A^2(D)$. 
Therefore $T$ extends uniquely to an isometric map
$$
T\colon A^2(\Omega)\longrightarrow A^2(D).
$$
Its range is closed, since $T$ is isometric; it is also dense in $A^2(D)$ by
the preceding observation. Hence $T$ is surjective, and therefore an
isometric isomorphism from $A^2(\Omega)$ to $A^2(D)$.

By \eqref{Equation *}, for $z,w\in D$,
$$
\frac{1}{\overline{\varphi(w)}}K_w^D(z)
=
\varphi(z)K_w^\Omega(z).
$$
Hence
\begin{equation}\label{unitary T}
T(f)=\varphi f|_D
\quad\hbox{on }D
\end{equation}
for every $f\in\mathcal E_\Omega$. 
Since $T$ is unitary and convergence in a Bergman space
implies locally uniform convergence, the same identity holds for every
$f\in A^2(\Omega)$. Consequently, by unitarity of $T$ again,
\begin{equation}\label{unitary multiplication identity}
\langle f,g\rangle_{A^2(\Omega)}
=
\langle \varphi f,\varphi g\rangle_{A^2(D)}
\qquad
\hbox{for all }f,g\in A^2(\Omega).
\end{equation}
We now apply a complex moment argument. Since $\Omega$ is bounded,
every monomial belongs to $A^2(\Omega)$. Taking
$f=z^\alpha$ and $g=z^\beta$ in
\eqref{unitary multiplication identity}, we obtain
$$
\int_\Omega z^\alpha\bar{z}^{\beta}\,dV
=
\int_D
z^\alpha\bar{z}^{\beta}|\varphi|^2\,dV
$$
for all multi-indices $\alpha,\beta\in\mathbb N^n$. Thus the compactly
supported measures
$
\chi_\Omega\,dV$
and $|\varphi|^2\chi_D\,dV
$
have the same complex moments. By uniqueness of the compactly supported
moment measures,
\begin{equation}\label{measure equality for D and Omega}
\chi_\Omega\,dV
=
|\varphi|^2\chi_D\,dV.
\end{equation}
Since $D\subseteq\Omega$, it follows that
$$
m(\Omega\setminus D)=0
$$
and $|\varphi|=1$ almost everywhere on $D$. By continuity,
$|\varphi|\equiv1$ on $D$, and hence, by the open mapping theorem,
\begin{equation}\label{varphi is unimodular constant}
\varphi\equiv c
\quad\hbox{on }D
\end{equation}
for some constant $c$ with $|c|=1$. Then the desired equation \eqref{First appearance of the functional equation of the Bergman kernel for lambda equal one} follows from \eqref{Equation *}.
Finally, by \eqref{unitary T} and
\eqref{varphi is unimodular constant}, the unitary operator $T$ is simply the restriction map, up to a unimodular constant:
$
T(f)=c\,f|_D.
$
Therefore the restriction map
$$
R\colon A^2(\Omega)\longrightarrow A^2(D),
\qquad
R(f)=f|_D,
$$
is a unitary isomorphism. In particular, every function in $A^2(D)$
extends holomorphically to a function in $A^2(\Omega)$. Together with
$m(\Omega\setminus D)=0$, this shows that $\Omega\setminus D$ is
Bergman-negligible in $\Omega$. The proof is complete.
\end{proof}

\section{Proofs of Theorem~\ref{Main result - bounded homogoeneous case}, its corollaries, and
Theorem~\ref{Lambda = 1 Theorem}}
\label{Subsection where Theorem 5 is proved}

We first prove the following more general version of
Theorem~\ref{Main result - bounded homogoeneous case},
namely, Theorem~\ref{Main theorem bounded homogeneous case - but with M a complex manifold that separates points}.
Theorem~\ref{Main result - bounded homogoeneous case} then follows as a
special case.

\begin{theorem}\label{Main theorem bounded homogeneous case - but with M a complex manifold that separates points}
Let $M$ be an $n$-dimensional Bergman-nondegenerate complex manifold, $n \geq 2$, with its Bergman metric given by $\omega_M$ and let $\Omega \subseteq \mathbb{C}^n$ be a bounded domain with a complete Bergman metric $\omega_{\Omega}$. Suppose there exists an
analytically Zariski dense subset $X\subseteq\Omega$ such that, for every
$p\in X$, the $\operatorname{Aut}(\Omega)$-orbit through $p$ contains a nonconstant
holomorphic curve.  Assume further that there exist an open connected set
$U\subseteq M$ and a holomorphic map $f:U\to\Omega$ such that
\begin{equation}\label{Our constant that will be replaceed by a conformal factor once again 23}
\omega_M=\Lambda f^*(\omega_\Omega),
\end{equation}
where $\Lambda$ is a positive function on $U$.  Then $f$ extends to a biholomorphism $F\colon M\to D$ onto a subdomain
$D\subseteq\Omega$ such that $\Omega\setminus D$ is a
Bergman-negligible subset of $\Omega$.
\end{theorem}

\begin{proof}
By Proposition~\ref{Conformal implies constant proposition in dimension greater than one},
$\Lambda$ is a positive constant $\lambda$ on $U$. 
Then by \eqref{Our constant that will be replaceed by a conformal factor once again 23},
$f^*(\lambda\omega_\Omega)=\omega_M$.
Since $M$ is Bergman-nondegenerate,  $A^2_{(n,0)}(M)$ separates points of $M$.
As in the proof of Theorem~\ref{The main Stein result}, it follows from Mok \cite{Mo12} (or by using Calabi's rigidity theorem) that  \(f\) extends to a biholomorphism $F$ from \(M\) onto
a subdomain $D$ of \(\Omega\).
By the biholomorphic invariance of the Bergman metric,
$F^*\omega_D=\omega_M$, while the construction of $F$ gives
$F^*(\lambda\omega_\Omega)=\omega_M$. Since $F\colon M\to D$ is
biholomorphic, it follows that
$\omega_D=\lambda\omega_\Omega$
on $D,$
or equivalently,
\begin{equation}\label{The functional Bergman equation yet againa}
i\partial\overline{\partial}\log K_D
=
\lambda i\partial\overline{\partial}\log K_\Omega
\quad\hbox{on }D.
\end{equation}
By the orbit assumption and
part~\textup{(1)} of
Proposition~\ref{The lemma that says in the BH case that lambda equals 1},
we have $\lambda=1$. Proposition~\ref{Theorem I'} then applies to
\eqref{The functional Bergman equation yet againa} and shows that
$\Omega\setminus D$ is Bergman-negligible in $\Omega$. This completes
the proof.
\end{proof}

\begin{proof}[Proof of Theorem~\ref{Main result - bounded homogoeneous case}]
Since every bounded domain in $\mathbb C^n$ is Bergman-nondegenerate,
Theorem~\ref{Main result - bounded homogoeneous case} follows immediately
from Theorem~\ref{Main theorem bounded homogeneous case - but with M a complex manifold that separates points}.
\end{proof}


We next prove
Corollary~\ref{The new main theorem 2s corollary is the bounded homogeneous domain}
and
Corollary~\ref{Corollary 8 - locally symmetric Bergman metrics and the image is a domain minus a Bergman negligible subset},
and show that Corollary~\ref{The new main theorem 2s corollary is the bounded homogeneous domain} remains valid under the weaker assumption that $M$ is
Bergman-nondegenerate.

\begin{proof}[Proof of Corollary~\ref{The new main theorem 2s corollary is the bounded homogeneous domain}]
Suppose first that $\Omega$ is a bounded homogeneous domain. Then its
Bergman metric is complete, and the $\operatorname{Aut}(\Omega)$-orbit
through every point is all of $\Omega$. In particular, every such orbit
contains a nonconstant holomorphic curve. Hence
Theorem~\ref{Main theorem bounded homogeneous case - but with M a complex manifold that separates points}
applies with $X=\Omega$.

Now suppose that $\Omega=\Omega_1\times\Omega_2$, where $\Omega_1$ is a
bounded homogeneous domain and $\Omega_2$ has a complete Bergman metric.
Since $\omega_\Omega=\omega_{\Omega_1} \oplus \omega_{\Omega_2}$, the Bergman
metric of $\Omega$ is the product metric and is complete. Moreover, for
every $(p_1,p_2)\in\Omega$, the automorphisms of the form
$\tau\times\operatorname{id}_{\Omega_2}$, with
$\tau\in\operatorname{Aut}(\Omega_1)$, show that the
$\operatorname{Aut}(\Omega)$-orbit through $(p_1,p_2)$ contains
$\Omega_1\times\{p_2\}$, and hence contains a nonconstant holomorphic
curve. Thus the same theorem again applies with $X=\Omega$.

Therefore
Corollary~\ref{The new main theorem 2s corollary is the bounded homogeneous domain}
follows. The same argument also shows that the corollary remains valid
under the weaker assumption that $M$ is a Bergman-nondegenerate complex
manifold.
\end{proof}

We now turn to Corollary \ref{Corollary 8 - locally symmetric Bergman metrics and the image is a domain minus a Bergman negligible subset}. 
Recall that a Riemannian manifold $M$ is locally symmetric if, for every
$p\in M$, there exist a neighborhood $U$ of $p$ and a local isometry
$s\colon U\to M$ such that $s(p)=p$ and $(ds)_p=-\operatorname{id}$.
Equivalently, if $\nabla$ denotes the Levi--Civita connection and $R$
the Riemannian curvature tensor, then $M$ is locally symmetric if and
only if $\nabla R=0$ (see \cite[Chapter IV]{He78}).
A bounded domain $\Omega\subseteq\mathbb C^n$ is called a bounded
symmetric domain if, for every $p\in\Omega$, there exists an involutive
holomorphic automorphism of $\Omega$ for which $p$ is an isolated fixed
point.

In \cite[Theorem~1.2]{LoPa26}, Loi--Palmieri proved that
if $M$ is a bounded pseudoconvex domain satisfying the hypotheses of Corollary \ref{Corollary 8 - locally symmetric Bergman metrics and the image is a domain minus a Bergman negligible subset}, then $M$ is
biholomorphic to a bounded symmetric domain in $\mathbb C^n$ with
possibly a closed pluripolar set removed.
Loi--Palmieri \cite[Conjecture~1]{LoPa26} then asked whether, without the pseudoconvexity assumption, $M$ is biholomorphic to a bounded symmetric domain
with a Bergman-negligible set removed. This question was also raised
independently by Zimmer \cite{Zi25} to the second author of this paper. Corollary \ref{Corollary 8 - locally symmetric Bergman metrics and the image is a domain minus a Bergman negligible subset} provides an affirmative answer to their question.


\begin{proof}[Proof of Corollary~\ref{Corollary 8 - locally symmetric Bergman metrics and the image is a domain minus a Bergman negligible subset}]
By the local structure theory of Hermitian symmetric spaces
\cite[Chapter~IV]{He78}, there exist a simply connected and irreducible
Hermitian symmetric space $(X,\theta_X)$, an open connected set
$V\subseteq X$, and a holomorphic isometry $g:V\to M$ onto an open
subset of $M$. Since $n\geq 2$, the local irreducibility assumption
rules out the possibility that $X$ is flat. We first show that $X$
is of noncompact type, using the boundedness of $M$.
Suppose, to the contrary, that $X$ is of compact type. Since $M$ is
bounded, $A^2(M)$ is infinite-dimensional. Choose an orthonormal basis
$\{\varphi_j\}_{j=1}^{\infty}$ of $A^2(M)$, and let
\[
    \beta_M:M\longrightarrow\mathbb{P}^{\infty},
    \qquad
    \beta_M=[\varphi_1, \varphi_2,\ldots],
\]
be the associated Bergman--Bochner map. Then
$\beta_M^*\omega_{FS}=\omega_M$. Since $X$ is simply connected,
Calabi's extension theorem extends the local holomorphic isometric
immersion $\beta_M\circ g$ to a holomorphic isometric immersion
\[
    \mathcal H:(X,\theta_X)\longrightarrow
    (\mathbb{P}^{\infty},\omega_{FS}).
\]

The components of $\mathcal H$ give rise to global holomorphic sections of the line bundle $L:=\mathcal H^*\mathcal O(1)$. Since $X$ is compact, $H^0(X,L)$ is finite-dimensional, and hence $\mathcal H(X)$ lies in a finite-dimensional projective subspace of $\mathbb P^\infty$; cf.~also \cite[Remark~4]{MTZ24}. Consequently, the functions $\varphi_j\circ g$ satisfy a nontrivial finite linear relation on $V$.
Since $g(V)$ is open in $M$, the
identity theorem gives the same linear relation among the functions
$\varphi_j$ on $M$, contradicting their orthonormality. Thus $X$
is of noncompact type.

By \cite[Chapter~VIII, Theorem~7.1]{He78}, $X$ admits a biholomorphic
realization $\Phi:X\to\Omega$ as an irreducible bounded symmetric
domain $\Omega\subseteq\mathbb{C}^n$. By irreducibility,
$\theta_X=\lambda\Phi^*\omega_\Omega$
for some constant $\lambda>0$. Setting $U=g(V)$ and $f=\Phi\circ g^{-1}:U\to\Omega$, we obtain
$\omega_M=\lambda f^*\omega_\Omega$ on $U$.
Since $\Omega$ is homogeneous, the conclusion follows from
Corollary~\ref{The new main theorem 2s corollary is the bounded homogeneous domain}.
\end{proof}

In the proof of
Theorem~\ref{Main theorem bounded homogeneous case - but with M a complex manifold that separates points},
a major step is to establish Proposition~\ref{The lemma that says in the BH case that lambda equals 1},
namely, to prove that $\lambda=1$. In
Theorem~\ref{Lambda = 1 Theorem}, however, this is already assumed.

\begin{proof}[Proof of Theorem~\ref{Lambda = 1 Theorem}]
First we consider when $n \geq 2$.  We repeat the argument in the first part of the proof of
Theorem~\ref{Main theorem bounded homogeneous case - but with M a complex manifold that separates points}
to conclude that $f$ extends to a biholomorphism $F$ from $M$ onto its
image $D:=F(M)\subseteq\Omega$, and that
\eqref{The functional Bergman equation yet againa} holds.
By \eqref{Exact condition}, however, the normalizing constant in the
present setting is $\lambda=1$.
Hence
Proposition~\ref{Theorem I'} applies directly and shows that
$\Omega\setminus D$ is Bergman-negligible in $\Omega$.
The same argument also shows that
Theorem~\ref{Lambda = 1 Theorem} remains valid under the weaker
assumption that $M$ is a Bergman-nondegenerate complex manifold.  When $n = 1$, this theorem was already resolved, under the weaker assumptions that $M$ is a Bergman-nondegenerate Riemann surface or a Stein Riemann surface, in Remark \ref{rmk:question21dimension1}.
\end{proof}

\section{Proofs of Theorem \ref{Thm:Main Hartogs Theorem} and Corollary \ref{Egg domain corollary}}\label{Sec: The last section for the hartogs domains over bounded homogeneous bases}

\subsection{A few preliminary lemmas}
\begin{lem}\label{lem:Omega-c-pseudo}
Let $B \subseteq \mathbb{C}^n$ be a bounded homogeneous domain and $\psi$ be a smooth strictly plurisubharmonic function on $B$. Let $c>0$ and consider the Hartogs domain 
$$
\Omega_c:=\left\{(z,\zeta)\in B\times\mathbb{C}^m:
 t(z,\zeta):=e^{\psi(z)}\|\zeta\|^2<c\right\}.
$$
Then $\Omega_c$ is hyperconvex and admits a complete Bergman metric. Moreover, the boundary hypersurface
\[
\Sigma_c:=\{(z, \zeta) \in B \times \mathbb{C}^m: t(z, \zeta) =c\} \subseteq \partial \Omega_c 
\]
is smooth and strongly pseudoconvex.
\end{lem}

\begin{proof}
Every bounded homogeneous domain is hyperconvex
\cite[p.~161]{KaOh07}.  Consequently, there exists a continuous
negative plurisubharmonic exhaustion function
$\rho_B\colon B\rightarrow(-\infty,0).
$
In particular, $\rho_B(z)\to 0$ whenever $z$ leaves every compact subset of
$B$.
Since $\psi(z)+\log\|\zeta\|^2$ is plurisubharmonic on
$B\times\mathbb{C}^m$ and the exponential function is increasing
and convex, $t=e^{\psi(z)}\|\zeta\|^2$ is plurisubharmonic
on $B\times\mathbb{C}^m$. 
Define
\[
    \rho_{\Omega_c}(z,\zeta)
    :=\max\left\{\rho_B(z),\,
                  e^{\psi(z)}\lVert\zeta\rVert^2-c\right\},
    \qquad (z,\zeta)\in\Omega_c.
\]
The function $\rho_{\Omega_c}$ is continuous, negative, exhaustive and plurisubharmonic.
Therefore $\Omega_c$ is 
hyperconvex.
By \cite[p.~606, Remark~(ii)]{NiPf05}, $\Omega_c$ admits a bounded
negative strictly plurisubharmonic exhaustion function.
Consequently, $\Omega_c$ admits a complete Bergman metric by \cite[Theorem~1]{Ch04}
(see also B{\l}ocki--Pflug
\cite{BlPf98} or Herbort \cite{Her99}). 
A standard argument shows that $\Sigma_c$ is strongly pseudoconvex.
\end{proof}

We also prove the following lemma about radial pluriharmonic functions.  Here  $U(m)$ denotes the unitary matrices on $\mathbb{C}^m$.

\begin{lem}\label{lem:radial}
Let 
\[
A_{r_1,r_2}^{m}:=\{\zeta\in\mathbb{C}^m:r_1<\|\zeta\|^2<r_2\},
\qquad 0\le r_1<r_2, \qquad m \in \mathbb{N}.
\]
Suppose $u$ is real-valued, pluriharmonic, and $U(m)$-invariant on $A_{r_1,r_2}^{m}$. Then
\begin{equation*}\label{eq:radial-classification}
u(\zeta)=
\begin{cases}
c\log\|\zeta\|^2+C,&m=1,\\[1mm]
C,&m\ge2,
\end{cases}
\end{equation*}
for real constants $c$ and $C$.
\end{lem}

\begin{proof}
By unitary invariance, $u(\zeta)=F(s)$ for some $F\in C^\infty((r_1,r_2))$, where $s=\|\zeta\|^2$. Its complex Hessian is
\begin{equation}\label{eq:radial-hessian}
\frac{\partial^2u}{\partial\zeta_j\partial\overline\zeta_k}
=F'(s)\delta_{jk}+F''(s)\overline\zeta_j\zeta_k.
\end{equation}
If $m\geq2$, fix any $\zeta\in A_{r_1,r_2}^{m}$ and choose
$v\in\mathbb C^m\setminus\{0\}$ such that
$\sum_{j=1}^m\overline\zeta_jv_j=0$. Since $u$ is pluriharmonic, \eqref{eq:radial-hessian} gives
$$
0=\sum_{j,k=1}^m
\frac{\partial^2u}{\partial\zeta_j\partial\overline\zeta_k}
v_j\overline v_k
=F'(s)\|v\|^2.
$$
Hence $F'(s)=0$ for every $s\in(r_1,r_2)$, so $u$ is constant.
If $m=1$, pluriharmonicity and \eqref{eq:radial-hessian} give
\[
F'(s)+sF''(s)=0,
\]
so $(sF'(s))'=0$. Thus $F'(s)=c/s$ and $F(s)=c\log s+C$.
\end{proof}

We will need to compare the Bergman kernel of the biholomorphic image $D$ from Theorem \ref{Thm:Main Hartogs Theorem} with the Bergman kernels of inner and outer pseudoconvex domains.  The outer pseudoconvex domain will be given by $\Omega$.  The inner one comes from a construction by the following lemma of Bell \cite{Be86}. For the comparison, we will need Lemma~\ref{lem:local-Bergman}.

\begin{lem}\cite[Lemma, Section 4]{Be86}\label{lem:inner-domain}
Let $\Omega \subseteq \mathbb C^n$ be a bounded pseudoconvex domain and $p \in \partial\Omega$ a smooth and strongly pseudoconvex boundary point. There is an arbitrarily small bounded smooth strongly pseudoconvex domain $G_p \subseteq \Omega$ such that there is a neighborhood $W$ of $p$ in $\mathbb C^n$ so that 
$G_p\cap W=\Omega\cap W.$
\end{lem}


\begin{lem}\label{lem:local-Bergman}
Let $U\subseteq\mathbb{C}^N$ be a bounded pseudoconvex domain, and let $p\in\partial U$ be 
a smooth and strongly pseudoconvex boundary point.
Then there are a neighborhood $W$ of $p$ and constants $C_1,C_2>0$ such that
\begin{equation}\label{eq:local-estimate}
C_1\,\delta_U(x)^{-(N+1)}
\le K_U(x,x)
\le C_2\,\delta_U(x)^{-(N+1)}
\end{equation}
for $x\in U\cap W$, where $\delta_U(x)$ denotes the Euclidean boundary distance function.
\end{lem}

Lemma~\ref{lem:local-Bergman} follows by combining
Lemma~\ref{lem:inner-domain} with the localization lemma in
\cite[p.~898]{Oh84} and the classical Bergman kernel estimates
of H\"ormander and Diederich; see \cite[Theorem~1]{Oh84}.


\subsection{Reduction to $\lambda = 1$}

\begin{pro}\label{thm:scaling}
Let $B$ and $\Omega$ be as in Theorem \ref{Thm:Main Hartogs Theorem}.   Suppose that $D$ is a domain with $D \subseteq \Omega$ and satisfies
\begin{equation}\label{eq:metric-scaling}
\omega_D=\lambda\,\omega_\Omega|_D,
\end{equation}
for some constant $\lambda>0$. Then $\lambda=1.$
\end{pro}
\begin{proof}
Let $\hbox{Isom}(B)$ be the group of holomorphic isometries of the base $B$ for the K\"ahler metric $\omega_{\psi}$. Then by the isometry property, 
for every $g\in \hbox{Isom}(B)$, the function $\psi\circ g-\psi$ is pluriharmonic.
Since $B$ is a bounded homogeneous domain, it is biholomorphic to a homogeneous Siegel domain and is, in particular, simply connected (cf.,  \cite{Ka71}).  
Consequently, there is a holomorphic function $h_g$ on $B$ such that
\begin{equation}\label{eq:cocycle-potential}
\psi(g(z))-\psi(z)=2\operatorname{Re}h_g(z).
\end{equation}
Recall that $t(z,\zeta):=e^{\psi(z)}\|{\zeta}\|^2$ so that $\Omega$ can be rewritten as
\begin{equation}\label{eq:Omega-def}
\Omega
=
\left\{(z,\zeta)\in B\times\mathbb{C}^m:
 t(z,\zeta)<1\right\}.
\end{equation}

\medskip
\noindent{\bf Step 1. Lifting automorphisms and the automorphic invariance.} We begin by lifting the isometries of $\hbox{Isom}(B)$ to $\Omega$. 

\begin{claim}\label{lem:lifts} For $g\in \hbox{Isom}(B)$ and $U\in U(m)$, define a lifted isometry of $g$ to $\Omega$ by
\begin{equation*}\label{eq:lifted-map}
\widetilde g_U(z,\zeta)
=
\left(g(z),e^{-h_g(z)}U\zeta\right).
\end{equation*}
Every map $\widetilde g_U$ is an automorphism of $\Omega$, and
$t\bigl(\widetilde g_U(z,\zeta)\bigr)=t(z,\zeta).$
Moreover, these lifted isometries form a group $\mathcal A(\Omega)$ under composition. This group acts transitively on every level set
$\Sigma_c:=\{t=c\}$, $0\le c<1$.
\end{claim}
\begin{proof}
Using \eqref{eq:cocycle-potential},
$$
t\bigl(\widetilde g_U(z,\zeta)\bigr)
=e^{\psi(g(z))-2\operatorname{Re}h_g(z)}\|U\zeta\|^2
=e^{\psi(z)}\|\zeta\|^2
=t(z,\zeta).
$$
Thus $\widetilde g_U$ preserves $\Omega$. It is a biholomorphism because $g$ is a biholomorphism and the fiber multiplier, $e^{-h_g(z)}$, never vanishes; its inverse is obtained from $g^{-1}$ and another nowhere-zero holomorphic fiber multiplier. It is routine to check that these lifted isometries form a group $\mathcal A(\Omega)$ under composition.
Let $(z,\zeta)$ and $(z',\zeta')$ lie in the same level $\Sigma_c$. Choose $g\in \hbox{Isom}(B)$ with $g(z)=z'$. By the definition of $t$ and \eqref{eq:cocycle-potential},
\[
\left\|e^{-h_g(z)}\zeta\right\|=\|{\zeta'}\|.
\]
When $c>0$, both vectors are nonzero, so a unitary matrix $U$ can be chosen such that
$e^{-h_g(z)}U\zeta=\zeta'$. When $c=0$, both fiber vectors vanish, and transitivity on $\Sigma_0$ follows from the transitivity of $\hbox{Isom}(B)$ on $B$. Hence for every $c$ with $0 \leq c < 1$, $\Sigma_c$ is one orbit.
\renewcommand{\qedsymbol}{$\blacksquare$}
\end{proof}

By Proposition \ref{The lemma that says in the BH case that lambda equals 1}, there is a domain $\widehat{D}$ such that $\widehat{D} \setminus D$ is Bergman-negligible, $\omega_{\widehat{D}}=\lambda\,\omega_{\Omega}|_{\widehat{D}}$ and $\widehat{D}$ is invariant under $\mathcal A(\Omega)$.  It thus suffices to show that $\lambda = 1$ for the new domain $\widehat{D}$.  In order to reduce notational complexity, without loss of generality, we will assume that $D = \widehat{D}$; that is, that $D$ is invariant under the action of $\mathcal A(\Omega)$.

\medskip
\noindent{\bf Step 2. Reduction of $D$ to a $t$-annular region.} We next describe $D$ in terms of the fiber-norm $t$.  Since by Claim \ref{lem:lifts} every level of $t$ is one orbit of $\mathcal A(\Omega)$, $D$ is a union of level sets.  Hence $D = t^{-1}(I)$ for some $I\subseteq[0,1)$. Openness of $D$ makes $I$ relatively open. Connectedness of $D$ makes $I$ an interval. Therefore, $D$ is the preimage under $t$ of a relatively open interval of $[0,1)$.  If $D\cap(B\times\{0\})\neq\emptyset$, then the $\mathcal A(\Omega)$-invariance of $D$ and the transitivity of $\mathcal A(\Omega)$ on $B\times\{0\}$ imply that $B\times\{0\}\subseteq D$. Since $B\times\{0\}$ contains a nonconstant holomorphic curve, part (2) of Proposition~\ref{The lemma that says in the BH case that lambda equals 1} gives $\lambda=1$. It therefore remains to consider the case where 
\begin{equation}\label{eq:D-is-annulus}
D = D_{a, b} := \{(z, \zeta): a<t(z, \zeta) <b\}
\end{equation}
for some $0\le a<b\le1$ and show that $\lambda = 1$.

\medskip
\noindent{\bf Step 3. A pluriharmonic annular function.}
Set
\begin{equation}\label{eq:P-def}
P:=\log K_{D}-\lambda\log K_\Omega.
\end{equation}
By \eqref{eq:metric-scaling}, $i\partial\overline\partial P=0$, so $P$ is a globally defined real-valued pluriharmonic function on $D$. For $U\in U(m)$, using \eqref{eq:Omega-def} and \eqref{eq:D-is-annulus}, the map
\[
R_U(z,\zeta)=(z,U\zeta)
\]
is an automorphism of both $\Omega$  and $D$. Its complex Jacobian has modulus one. The transformation law of the Bergman kernel therefore gives
$$
K_{D}(R_Ux,R_Ux)=K_{D}(x,x),
\qquad
K_\Omega(R_Ux,R_Ux)=K_\Omega(x,x), \quad x = (z, \zeta).
$$
Consequently, by \eqref{eq:P-def},
$P(z,U\zeta)=P(z,\zeta).$
Fix $z_0\in B$. By \eqref{eq:D-is-annulus}, the fiber over $z_0$ in $D$ is the annular domain 
$$
A_{z_0}^{a, b} := \{\zeta \in \mathbb{C}^m: a e^{-\psi(z_0)}<\|\zeta\|^2< be^{-\psi(z_0)}\}.
$$
Applying Lemma~\ref{lem:radial} to the annular pluriharmonic function $P(z_0,\cdot): A_{z_0}^{a, b} \to \mathbb{R}$, we conclude that when $m \geq 2$, $P(z_0, \cdot)$ is constant.  On the other hand, when $m = 1$, for some constants $c$ and $C$, it is given by $P(z_0, \cdot) = c\log\|\cdot\|^2 + C$.  In both cases,
\begin{equation}\label{Boundedness of P}
\lim_{t(z_0, \cdot) \to b^{-}} P(z_0, \cdot) = \lim_{\|\zeta\|^2 \to (be^{-\psi(z_0)})^{-}} P(z_0, \zeta)  < \infty.
\end{equation}

\medskip

\noindent{\bf Step 4. Deducing $b = 1$.} We will show that $b = 1$ in \eqref{eq:D-is-annulus}. Suppose, to the contrary, that $b<1$. Then
$D_{a,b}\subseteq\Omega_b.$
Monotonicity of the diagonal Bergman kernel under domain inclusion gives
\begin{equation}\label{eq:monotonicity-b}
K_{D_{a,b}}(x,x)\ge K_{\Omega_b}(x,x),
\qquad x\in D_{a,b}.
\end{equation}
Fix $z_0\in B$ and choose $\zeta_0$ with
$t(z_0,\zeta_0)=b$. Let $x_j=(z_0,r_j\zeta_0)$, where $0<r_j<1$ and $r_j\to1$. Then $x_j\in D_{a,b}$ for sufficiently large $j$, and $x_j\to p:=(z_0,\zeta_0)\in\Sigma_b$.
By Lemma~\ref{lem:Omega-c-pseudo}, $\Omega_b$ is pseudoconvex and $p$ is a smooth strongly pseudoconvex boundary point. Lemma~\ref{lem:local-Bergman} then yields
$
K_{\Omega_b}(x_j,x_j)\to +\infty.
$
By \eqref{eq:monotonicity-b},
\begin{equation}\label{eq:KD-blows}
K_{D_{a,b}}(x_j,x_j)\longrightarrow+\infty.
\end{equation}
On the other hand, $p$ is an interior point of $\Omega=\{t<1\}$ because $b<1$. Hence $K_\Omega(x,x)$ is bounded in a neighborhood of $p$. It follows from \eqref{eq:P-def} and \eqref{eq:KD-blows} that
$
P(x_j)\longrightarrow+\infty.
$
This contradicts \eqref{Boundedness of P}, because the sequence approaches the outer radius inside the fixed fiber over $z_0$. 
Thus, $b = 1$ and \eqref{eq:D-is-annulus} becomes
\begin{equation}\label{Updated description of D}
D = \{(z, \zeta): a < t(z, \zeta) < 1\}.
\end{equation}

\medskip

\noindent{\bf Step 5. Squeezing Bergman kernels.} Fix 
$$
p:=(z_0,\zeta_0)\in\Sigma_1 = \{(z, \zeta) \in \partial \Omega: t(z, \zeta) = 1\}.
$$ 
By Lemma \ref{lem:inner-domain},
there exist a small smoothly bounded strongly pseudoconvex domain $G_p \subseteq D$ and a small neighborhood $W$ of $p$ in $\mathbb{C}^{m+n}$ such that $G_p \cap W = \Omega \cap W$. Since $
G_p\subseteq D\subseteq\Omega$,
by monotonicity of the Bergman kernel,
\begin{equation}\label{eq:squeeze}
K_\Omega(x,x)
\le K_{D}(x,x)
\le K_{G_p}(x,x), \quad x \in G_p.
\end{equation}

Let $N = n+m$. The domain $\Omega$ is pseudoconvex by Lemma~\ref{lem:Omega-c-pseudo}, and $p$ is a smooth strongly pseudoconvex boundary point. Applying Lemma~\ref{lem:local-Bergman} to $\Omega$ and $G_p,$ respectively, we get
\begin{equation}\label{eq:K-Omega-order}
K_\Omega(x,x)\asymp\delta_\Omega(x)^{-(N+1)}; \quad  K_{G_p}(x,x)\asymp\delta_{G_p}(x)^{-(N+1)}
\end{equation}
near $p$. 
Because $G_p$ and $\Omega$ have the same smooth boundary germ near $p$, their boundary distances are the same for $x$ near $p$. Combining \eqref{eq:squeeze}--\eqref{eq:K-Omega-order}, we obtain
$K_{D}(x,x)\asymp\delta(x)^{-(N+1)},$
where $\delta$ may be taken to be either $\delta_\Omega$ or $\delta_{G_p}$ near $p$. In logarithmic form, these imply
\begin{align}
\log K_D(x,x)
&=-(N+1)\log\delta(x)+O(1),\nonumber\\
\log K_\Omega(x,x)
&=-(N+1)\log\delta(x)+O(1).\nonumber
\end{align}
Therefore, by \eqref{eq:P-def},
\begin{equation}\label{eq:P-asymptotic}
P(x)
=(N+1)(\lambda-1)\log\delta(x)+O(1).
\end{equation}
Approach $p$ along 
$x_r=(z_0,r\zeta_0)$ with $r\to1^-$. Along this path, $t(x_r)\to1^-$ and $\delta(x_r)\to0$. Equation \eqref{Boundedness of P} says that $P(x_r)$ remains bounded. Since
$\log\delta(x_r)\to-\infty$, equation \eqref{eq:P-asymptotic} is compatible with boundedness only if $\lambda-1=0$. Thus $\lambda=1$.
\end{proof}

\subsection{Proof of Theorem \ref{Thm:Main Hartogs Theorem} and Corollary \ref{Egg domain corollary}}

\begin{proof}[Proof of Theorem \ref{Thm:Main Hartogs Theorem}]
The proof follows the proof of Theorem \ref{Main theorem bounded homogeneous case - but with M a complex manifold that separates points} except that one uses Lemma \ref{lem:Omega-c-pseudo} to guarantee that the Bergman metric of $\Omega$ is complete and instead of using Proposition \ref{The lemma that says in the BH case that lambda equals 1}, one uses Proposition \ref{thm:scaling} to show that $\lambda = 1$.  As Theorem \ref{Main theorem bounded homogeneous case - but with M a complex manifold that separates points} holds under the weaker assumption that $M$ is a Bergman-nondegenerate complex manifold, the present theorem also holds under the Bergman-nondegenerate complex manifold condition on $M$.
\end{proof}

\begin{proof}[Proof of Corollary \ref{Egg domain corollary}]

For this setting, let
$$
B=\mathbb{B}^n=\{z \in \mathbb{C}^n: \|z\|<1 \}, \quad \psi=-\frac{1}{p}\log(1-\|z\|^2), \qquad p>0.
$$
Notice that $\omega_{\psi}=i\partial\overline{\partial}\psi$ is a homogeneous K\"ahler metric with its holomorphic isometry group equal to $\hbox{Aut}(\mathbb{B}^n)$, because it is proportional to the Bergman metric.
The corollary now follows from Theorem \ref{Thm:Main Hartogs Theorem} using the above $B$ and $\psi$.  Since Theorem \ref{Thm:Main Hartogs Theorem} remains valid under the weaker assumption that $M$ is a
Bergman-nondegenerate complex manifold, so does this corollary. 
\end{proof}

\begin{remark}\label{Final remark saying the egg domain corollary only follows from the last theorem}
The automorphism groups of the egg domains in Corollary \ref{Egg domain corollary} are explicitly known; see \cite[Theorem A, Case I]{Ko17} and \cite{Ko14}.  For the nonhomogeneous egg domains, the only automorphism orbit containing a nonconstant holomorphic curve is $\mathbb B^n\times\{0\}$, where $0\in\mathbb C^m$.  Thus, for such a domain, there is no analytically Zariski dense subset $X\subseteq\Omega$ such that, for every $p\in X$, the orbit of $p$ under $\operatorname{Aut}(\Omega)$ contains a nonconstant holomorphic curve. These nonhomogeneous egg domains therefore do not satisfy the hypotheses of Theorems~\ref{Main result - bounded homogoeneous case} and~\ref{Main theorem bounded homogeneous case - but with M a complex manifold that separates points}. This shows that Theorem~\ref{Thm:Main Hartogs Theorem} is the only theorem in this paper from which Corollary~\ref{Egg domain corollary} follows.
\end{remark}

\bibliographystyle{amsplain}
\bibliography{bibliography}

\fontsize{11}{9}\selectfont

\vspace{0.5cm}

\noindent pebenfelt@ucsd.edu;

 \vspace{0.2 cm}

\noindent Department of Mathematics, University of California San Diego, La Jolla, CA 92093, USA

\vspace{0.6 cm}

\noindent john.treuer@sjsu.edu;

\vspace{0.2 cm}

\noindent Department of Mathematics, San Jos\'e State University, San Jos\'e, CA 95192, USA

\vspace{.6 cm}

\noindent m3xiao@ucsd.edu;

 \vspace{0.2 cm}

\noindent Department of Mathematics, University of California San Diego, La Jolla, CA 92093, USA

\end{document}